\documentclass{article}
\usepackage[utf8]{inputenc}
\usepackage{amssymb,amsmath,latexsym}
\usepackage{enumerate}
\usepackage[english]{babel}
\usepackage{amsthm}
\usepackage{verbatim}
\usepackage{graphicx}
\usepackage{color}
\usepackage{xcolor}
\usepackage{tikz}
\usepackage{float}
\usetikzlibrary{calc, intersections, through, backgrounds}
\usetikzlibrary{shapes.geometric}
\tikzstyle{point}=[ball color=white, circle, draw=black, inner sep=0.1cm]
\tikzstyle{redball}=[ball color=red, circle, draw=black, inner sep=0.1cm]
\tikzstyle{blueball}=[ball color=blue, circle, draw=black, inner sep=0.1cm]
\tikzstyle{blackball}=[ball color=black, circle, draw=black, inner sep=0.1cm]
\tikzstyle{greenball}=[ball color=green, circle, draw=black, inner sep=0.1cm]
\tikzstyle{violetball}=[ball color=violet, circle, draw=black, inner sep=0.1cm]
\tikzstyle{cyanball}=[ball color=cyan, circle, draw=black, inner sep=0.1cm]
\tikzstyle{brownball}=[ball color=brown, circle, draw=black, inner sep=0.1cm]
\tikzstyle{orangeball}=[ball color=orange, circle, draw=black, inner sep=0.1cm]
\tikzstyle{yellowball}=[ball color=yellow, circle, draw=black, inner sep=0.1cm]
\tikzstyle{whiteball}=[ball color=white, circle, draw=black, inner sep=0.1cm]
\tikzset{redballhollow/.style={circle,draw=red,line width=1pt,fill=white,inner sep=0.1cm}}
\usepackage{caption}
\usepackage{subcaption}
\usetikzlibrary{arrows.meta}
\usepackage{authblk}
\usepackage{orcidlink}
\usepackage{mathtools}
\usepackage[vlined,ruled,linesnumbered]{algorithm2e}

\newcommand{\verteq}{\rotatebox{90}{$\,=$}}
\newcommand{\equalto}[2]{\underset{\overset{\verteq}{#2}}{#1}}

\newtheorem{theorem}{Theorem}[section]

\newtheorem{corollary}[theorem]{Corollary}

\newtheorem{definition}[theorem]{Definition}

\newtheorem{lemma}[theorem]{Lemma}

\newtheorem{proposition}[theorem]{Proposition}
\newtheorem{remark}[theorem]{Remark}

\newcommand{\ccap}[2]{\mathrm{cap}^{\mathbin{\Box}}_{#1}(#2)}

\DeclareMathOperator{\BoxProduct}{\mathbin{\Box}}
\DeclareMathOperator{\orb}{\mathrm{orb}}

\DeclareMathOperator{\lev}{\mathrm{lev}}
\DeclareMathOperator{\brig}{\mathrm{brig}}

\newcommand{\cvSpan}[1]{\sigma^{\BoxProduct}_V(#1)}

\SetKwData{MaxM}{max\_M}
\SetKwData{CurrentM}{current\_M}
\SetKwData{Visited}{visited}
\SetKwData{MaxDeg}{$\Delta$}
\SetKwData{Undefined}{undefined}
\SetKwFunction{Brig}{brig}
\SetKwData{Boundary}{boundary}
\SetKwData{True}{True}
\SetKwData{False}{False}
\SetKwData{Current}{current}
\SetKwData{Neighbor}{neighbor}
\SetKwData{Queue}{Q}

\SetKwFunction{Enqueue}{enqueue}
\SetKwFunction{Dequeue}{dequeue}
\SetKwFunction{Append}{append}

\title{Results on Cartesian $1$-capacity of graphs}

\author[1,2]{Mateja Grašič \orcidlink{0000-0003-2684-7632}\thanks{mateja.grasic@um.si}}
\author[3]{Christopher Mouron \orcidlink{0009-0008-5122-7519}\thanks{mouronc@rhodes.edu}}
\author[4]{Aljoša Šubašić \orcidlink{0000-0002-6943-0856}\thanks{aljsub@pmfst.hr}}
\author[1,2]{Andrej Taranenko \orcidlink{0000-0002-2438-0496}\thanks{ andrej.taranenko@um.si}}
\author[4]{Tanja Vojković \orcidlink{0000-0002-8160-7757}\thanks{tanja@pmfst.hr}}

\affil[1]{Faculty of Natural Sciences and Mathematics, University of Maribor, Slovenia}
\affil[2]{Institute of Mathematics, Physics and Mechanics, Ljubljana, Slovenia}
\affil[3]{Department of Mathematics and Statistics, Rhodes College, Memphis, TN, 38112 USA}
\affil[4]{Faculty of Science, University of Split, Croatia}

\date{}

\begin{document}

\maketitle

\begin{abstract}
    The concept of graph capacity extends graph span by considering the maximum number of agents that can simultaneously traverse a graph while preserving a prescribed minimum distance. We study the Cartesian $1$-capacity, corresponding to the movement in which exactly one agent moves at each step. We establish general lower and upper bounds based on structural graph properties and derive exact results for trees. Our work highlights the role of branching and bridge structures in determining the Cartesian $1$-capacity and provides new insights into collision-free multi-agent motion on graphs.
\end{abstract}

\section{Background and motivation}

The concept of graph capacity originated from the study of graph span, introduced by Banič and Taranenko \cite{BaTa23} as a graph-theoretic counterpart to the classical topological notion of span of a continuum. The original motivation arose from scenarios where moving agents must preserve a minimum mutual distance---a constraint that gained particular relevance during the COVID-19 pandemic. In this framework, two players traverse a graph to cover all vertices or edges while adhering to one of three movement models: traditional (players independently move or stay), active (both must move), or lazy (exactly one moves per step). These restrictions define the strong, direct, and Cartesian spans, respectively. Since their introduction, a growing body of literature has extensively explored these parameters, yielding polynomial-time algorithms, bounding the differences between span variants, and establishing exact values and structural characterizations for several graph families, including trees, paths, cycles, and multilayered networks \cite{BaTa23, DrMiTa25, Erceg23, GrMoTa25, SuVo24, SuVo25}.

Building upon this foundation, Grašič et al.\ \cite{gra26spancapacities} recently generalized this framework by introducing the notion of \emph{$d$-capacity}. While graph spans quantify the maximum distance that two players can maintain while traversing a graph, $d$-capacities address the complementary topological question: what is the largest number of players that can simultaneously traverse a network while preserving a prescribed minimum distance $d$? Three variants of the parameter---corresponding to strong, direct, and Cartesian movement---were introduced. Initial research established exact capacity values for basic topologies like paths and cycles, provided bounds for bipartite graphs, characterized graphs attaining the maximum possible $1$-capacity, and uncovered strong connections to classical concepts like graph factorizations, Hamiltonicity, and connectivity \cite{gra26spancapacities}. In this paper, we significantly expand upon these findings by investigating the Cartesian $1$-capacity of graphs (representing lazy---where exactly one person moves at a time---movement at distance $1$), with a particular focus on establishing exact structural bounds for trees.

Given our agenda, it is worth noting a historical and mechanical connection between the Cartesian capacity, specifically when $\ccap{1}{G} = |V(G)| - 1$, and the classical movers' problem on graphs. In 1984, Kornhauser, Miller, and Spirakis \cite{Kornhauser1984} established the foundational mechanics of coordinating pebbles moving one at a time to adjacent unoccupied vertices. Mathematically, their framework perfectly mirrors our patient movement rules for $n-1$ players utilizing a single vacancy. Following their seminal work, decades of research heavily focused on algorithmic reachability and the computational complexity of navigating bottlenecks on trees \cite{Auletta1999}. 

Our work fundamentally shifts this perspective. Rather than employing algorithmic simulation to ask \textit{if} a specific puzzle can be solved, the Cartesian $1$-capacity defines the absolute maximum topological density of agents (number of players) a tree can support while maintaining total systemic mobility (visiting every vertex). By doing so, we prove that routing feasibility is intrinsically tied to a topological graph invariant ($\brig(T)$), bypassing the need for state-by-state algorithmic checks. 

Furthermore, our definitions formally unify several disparate computational bottlenecks recently identified across modern Multi-Agent Path Finding (MAPF) and pebble motion literature. In 2024, Ardizzoni et al.\ \cite{Ardizzoni2024} demonstrated that the complexity of routing on trees is bottle-necked by the maximum length of so-called \emph{corridors} (degree-2 paths), which are structurally identical to the $M$-bridges defined in this paper. Similarly, recent parametrized MAPF algorithms rely on identifying \emph{havens} (vertices with degree $3$ or higher) to temporarily store agents and resolve traffic \cite{Deligkas2025}. Our \emph{branch-vertex $N$-dense} property provides a mathematically rigorous formalization of this mechanism. Beyond theoretical path-finding, these invariants have immediate practical applications. For example, recent operations research models the Train Unit Shunting Problem as pebble motion on tree-structured railway yards \cite{hanou2023moving}. Our invariants could be directly applied to calculate the absolute theoretical limits on the number of train units a yard can support. Moreover, evaluating optimal shortest move sequences for pebble motion on trees remains an active area of investigation \cite{Nakamigawa2025}; by formally bounding total mobility via unbranched corridors ($\brig(T)$), our framework provides a geometric tool to establish tighter bounds on solution sequences.

The importance of these structural limits is further reinforced by parallel research in token-swapping and agent dispersion. Aichholzer et al.\ \cite{Aichholzer2022} and Biniaz et al.\ \cite{dmtcs:8383} highlighted the difficulty of token swapping on trees, emphasizing that the problem is notoriously NP-hard on subdivided stars. Because a subdivided star is topologically equivalent to a tree with many arbitrarily long $M$-bridges originating from a single central vertex, our theorems explicitly identify the geometric features responsible for these algorithmic failures. Finally, our approach parallels recent advancements in the near-linear time dispersion of mobile agents \cite{Sudo-et-al-2026}. While dispersion algorithms focus on the mechanics of untangling densely packed agents to unique nodes, the Cartesian $1$-capacity dictates the absolute topological threshold at which any such routing remains structurally feasible.

The paper is organized as follows. Section \ref{sec:definitions} presents the necessary definitions and preliminary concepts concerning patient marches, tours, graph spans, and Cartesian capacity. In Section \ref{sec:bounds} we investigate the Cartesian $1$-capacity of graphs. We first establish a general lower bound for branch-vertex $N$-dense graphs and then derive an upper bound based on the existence of  $M$-bridge structures. These results are subsequently combined in Section \ref{sec:trees}, to obtain exact values and structural characterizations of the Cartesian 1-capacity for trees. Section \ref{sec:conslusion} provides some ideas for further work.

\section{Basic definitions and terminology}\label{sec:definitions}

In this section we establish main terminology and definitions used throughout this paper. All graphs considered in this paper are finite, connected, simple, and undirected. For any graph theoretic notions not explicitly defined here, see \cite{west2001introduction}. Most of definitions used here were already defined in \cite{gra26spancapacities}, we restate them for completeness of this paper.

Let $G$ be a connected graph and $\ell$ be a non-negative integer. A \emph{weak walk} $W$ on $G$ is a sequence of vertices $w_0, w_1, \ldots, w_\ell$ in $V(G)$ such that $w_i w_{i+1} \in E(G)$ or $w_i = w_{i+1}$ for all $i\in\{0, 1, \ldots, \ell-1\}$. We may also call it a weak $\ell$-walk. Note that if $w_i w_{i+1} \in E(G)$ for every $i\in\{0, 1, \ldots, \ell-1\}$, then this is the common notion of a \emph{walk} in $G$. Thus, every walk is also a weak walk. Note, $\ell$ may be 0, hence a (weak) $0$-walk is just one vertex. Moreover, since we consider (weak) walks parametrized by time, we may refer to indices of the walks as to \emph{points in time}. For this reason we also call a function $f: \{0,1,\ldots, \ell\}\rightarrow V(G)$ with $f(t)f(t+1)\in E(G)$, for any $t\in\{0,1,\ldots, \ell-1\}$, a walk. Similarly, we call a function $f: \{0,1,\ldots, \ell\}\rightarrow V(G)$ with $f(t)=f(t+1)$ or $f(t)f(t+1)\in E(G)$, for any $t\in\{0,1,\ldots, \ell-1\}$, a weak walk. 

\begin{definition}[\cite{BaTa23}]\label{def:distance-tracks}
Let $G$ be a connected graph, $\ell$ a non-negative integer and $f,g:\{0,\ldots,\ell\}\rightarrow V(G)$ functions. The \emph{distance between $f$ and $g$} is defined as
\[
m_G(f,g)= \min\{d_G( f(t), g(t)) \mid t \in \{0,\ldots,\ell\}\}.
\]
\end{definition}

\begin{definition}[\cite{gra26spancapacities}]\label{def:distance-multiple-tracks}
Let $G$ be a connected graph, $p\geq 2$ a natural number and $\ell$ a non-negative integer. Let $f_1, f_2, \ldots, f_p:\{0,\ldots,\ell\}\rightarrow V(G)$ be functions. The distance between $f_1, f_2, \ldots, f_p$ is 
\[
    m_G(f_1, f_2, \ldots, f_p) = \min\{ m_G(f_i, f_j) \mid i,j\in \{1, 2, \ldots, p\} \text{ and } i\not=j  \}.
\]
\end{definition}

\begin{definition}[\cite{gra26spancapacities}]
Let $G$ be a connected graph, $p$ a natural number and $\ell$ a non-negative integer. $F=\{f_1,\ldots,f_p\}$ is called a \emph{(weak) $\ell$-march} on $G$ if each $f_i:\{0,1,\ldots,\ell\}\rightarrow V(G)$ is a (weak) walk on $G$. Also, $F(t):=\{f_1(t),\ldots,f_p(t)\}$ is called the \emph{$t$\textsuperscript{th} stage of $F$.} The number $\ell$ is called the length of $F$.
\end{definition}

When we say that $F=\{f_1,\ldots,f_p\}$ is a (weak) march on $G$, this means that there exists a non-negative integer $\ell$ such that $F$ is a (weak) $\ell$-march on $G$.

For a (weak) march $F=\{f_1,\ldots,f_p\}$ on some connected graph $G$, we define 
$$m_G(F):=m_G(f_1, f_2, \ldots, f_p).$$ Also, if $p=1$ define $m_G(F):=\infty$.

\begin{definition}[\cite{gra26spancapacities}]
Let $G$ be a connected graph, and $F=\{f_1,\ldots,f_p\}$ a (weak) march on $G$. We say that $F$ is a \emph{collision-free} (weak) march if $m_G(F)\geq 1$.
\end{definition}

\begin{definition}[\cite{gra26spancapacities}]\label{def:orbit}
Let $G$ be a connected graph and let $F=\{f_1,\ldots,f_p\}$ be a collision-free (weak) $\ell$-march on $G$. If $v=f_s(0)\in\{f_1(0),\ldots,f_p(0)\}$, then the
\emph{orbit of $v$ under $F$} is 
\[\orb(F,v)=\orb(f_s,v)=\{f_s(0),f_s(1),\ldots,f_s(\ell)\}.\]
\end{definition} 

\begin{definition}[\cite{gra26spancapacities}]\label{def:extension}
Let $G$ be a connected graph, $p$ a natural number, and let $F_1=\{f^1_1,\ldots,f^1_p\}$ be a (weak) $\ell_1$-march  on $G$ and $F_2=\{f^2_1,\ldots,f^2_p\}$ be a (weak) $\ell_2$-march on $G$ such that $F_1(\ell_1)=F_2(0)$. An \emph{extension of $F_1$ with $F_2$} is a (weak) $(\ell_1+\ell_2)$-march $F_2F_1=\{f_1,\ldots,f_p\}$ on $G$ defined by 
\[
f_s(i) = \begin{cases}
            f^1_s(t), & \text{if } t\in \{0,\ldots,\ell_1\},\\
            f^2_{k_s}(t-\ell_1), & \text{if } t\in \{\ell_1+1,\ldots,\ell_1+\ell_2\}
         \end{cases},
\]
where $k_s$ is the index such that  $f^1_s(\ell_1)=f^2_{k_s}(0)$. Inductively, we define $F_n F_{n-1} \ldots F_2 F_1=F_n(F_{n-1} \ldots F_2 F_1)$ in the same manner.
\end{definition}

\begin{remark}[\cite{gra26spancapacities}]\label{rem:extension-cf}
	If both $F_1=\{f^1_1,\ldots,f^1_p\}$ and $F_2=\{f^2_1,\ldots,f^2_p\}$ are collision-free (weak) $\ell_1$- and $\ell_2$-marches, respectively, then $F_2F_1$ is also a collision-free (weak) $(\ell_1+\ell_2)$-march.
\end{remark}

\begin{definition}[\cite{gra26spancapacities}]\label{def:reverse-march}
	Let $G$ be a connected graph. Let $F=\{f_1,\ldots,f_p\}$ be a (weak) $\ell$-march on $G$. The \emph{reverse of $F$}, denoted by $F^{-1}=\{\widehat{f}_1,\ldots,\widehat{f}_p\}$, 
	is the (weak) $\ell$-march on $G$ defined by $\widehat{f}_s(t)=f_s(\ell-t)$ for all $s\in \{1,\ldots,p\}$ and $t\in\{0,\ldots,\ell\}$.
\end{definition}

\begin{definition}[\cite{gra26spancapacities}]\label{def:patient-march}
Let $G$ be a connected graph and $F=\{f_1,\ldots,f_p\}$ be a weak $\ell$-march on $G$. We say that $F$ is \emph{patient} if for each $t\in\{0,\ldots,\ell-1\}$ there exists $i\in\{1, 2, \ldots, p\}$ such that $f_i(t)f_i(t+1)\in E(G)$ and $f_{j}(t)=f_{j}(t+1)$ for each $j\in\{1, 2, \ldots, p\}\setminus\{i\}$.
\end{definition}

\begin{remark}[\cite{gra26spancapacities}]
If $F=\{f_1\}$ is a patient $\ell$-march on a connected graph $G$, then Definition \ref{def:patient-march} implies that $f_1$ is a walk on $G$.
\end{remark}

\begin{definition}[\cite{gra26spancapacities}]\label{def:dual-march}
 Let $G$ be a connected graph of order $n$. Suppose that $F=\{f_1,\ldots,f_p\}$ is a patient $\ell$-march on $G$. The \emph{dual of $F$} is a patient $\ell$-march $D_F=\{g_1,\ldots,g_{n-p}\}$ on $G$ such that
 $\{f_1(t),\ldots,f_p(t)\}\cup \{g_1(t),\ldots,g_{n-p}(t)\}=V(G)$, for any $t\in\{0,1,\ldots,\ell\}$. 
\end{definition}

With respect to Definition \ref{def:dual-march}, we think of $\{f_1(t),\ldots,f_p(t)\}$ as the \emph{occupied vertices} and $\{g_1(t),\ldots,g_{n-p}(t)\}$ as the \emph{vacant vertices} of $G$ at stage $t$. Note that this implies that $f_i(t)\not=g_j(t)$ for all applicable $i,j,t$. 

Also note that if a patient march $F$ on $G$ is defined, then its dual $D_F$ is automatically defined. Also, $F=D_{D_F}$ is the dual of $D_F$, so it follows that if $D_F$ is defined, then $F$ is also defined. 

The (weak) tracks, which are (weak) walks through all the vertices of a given graph, were defined in \cite{Erceg23}, but we state them here for clarity. 

\begin{definition}\cite{Erceg23} Let $G$ be a connected graph and $\ell$ a non-negative integer. We say that a (weak) $\ell$-walk $f : \{0,\ldots,\ell\} \rightarrow V(G)$ is a \emph{(weak) $\ell$-track on $G$} if $f$ is surjective.
\end{definition}

\begin{definition}[\cite{gra26spancapacities}]
    Let $G$ be a connected graph and let $F=\{f_1,\ldots,f_p\}$ be a (weak / patient) $\ell$-march on $G$. If $f_i$ is a (weak) $\ell$-track for all $i\in \{1,\ldots,p\}$ we say that $F$ is a \emph{(weak / patient) $\ell$-tour on $G$}.
\end{definition}
   
\begin{remark}
We follow the terminology established in \cite{gra26spancapacities}. Specifically, for a connected graph $G$, where the order of the domain is not important for the context, we say that a function $f$ is a (weak) track on $G$ meaning that there exists a non-negative integer $\ell$ such that $f$ is a (weak) $\ell$-track. Similarly, we say that $F$ is a (weak / patient) tour meaning that there exists a non-negative integer $\ell$ such that $F$ is a (weak / patient) $\ell$-tour.
\end{remark}

We now restate the definition of the Cartesian vertex span. This was introduced in \cite{BaTa23}. 

\begin{definition}[\cite{BaTa23}]\label{def:allSpans}
Let $G$ be a connected graph. The \emph{Cartesian vertex span} of the graph $G$, denoted $\cvSpan{G}$, is $$\cvSpan{G} = \max \{ m_G(f,g)  \mid  \ell\in \mathbb{Z}_{\ge 0} \text{ and } \{f,g\} \text{ is a patient $\ell$-tour on $G$} \}.$$
\end{definition}

Now, we give the definitions of Cartesian $d$-capacity corresponding to the Cartesian vertex span. 

\begin{definition}[\cite{gra26spancapacities}]\label{def:Cartesian-capacity}
    Let $G$ be a non-trivial graph different from a path and $d \leq \sigma^\square_V(G)$ a natural number. We define the \emph{Cartesian $d$-capacity} of graph $G$, denoted by $\ccap{d}{G}$, as the maximum integer $c$ such that there exists a patient $\ell$-tour $F=\{f_1,f_2,\ldots,f_c\}$ on $G$ satisfying $m_G(F)= d$.
    If $d=\sigma^\square_V(G)$, we call it the \emph{Cartesian capacity} and denote it by $\ccap{}{G}$.
\end{definition}

Note, in Definition \ref{def:Cartesian-capacity} we require the graph to be non-trivial and not a path. This is due to the fact that the Cartesian vertex span in these cases equals to 0. If we were to allow for $d$ to equal zero in the above definitions, those corresponding capacities would be unbounded and of no interest for research. Hence, all observed graphs with respect to Cartesian $d$-capacity have the Cartesian vertex span at least 1, so $d$ being a natural number less than or equal to the corresponding span is well-defined. 

\section{Bounds on Cartesian 1-capacity}\label{sec:bounds}

We start with some definitions needed for this section. 
Let $G$ be a graph and $S\subseteq V(G)$. We use $G[S]$ to denote the subgraph of $G$ induced by the vertices $S$, and $G-S$ to denote the graph $G[V(G)\setminus S]$. When $S=\{v\}$ we will use the notation $G-v$ instead of $G-\{v\}$. If $e\in E(G)$, we use $G-e$ to denote the graph defined by $V(G-e)=V(G)$ and $E(G-e)=E(G)\setminus\{e\}$.

\begin{proposition}\label{together}
	Let $G$ be a connected graph, $H\subseteq V(G)$ such that $p=|H|$. Suppose for every $v\in H$ there exists a patient collision-free $\ell$-march $F=\{f_1,\ldots,f_p\}$ such that $F(0)=H$ and $\orb(F,v)=V(G)$. Then $\ccap{1}{G}\geq p$.
\end{proposition}

\begin{proof} Let $\{v_1,\ldots,v_p\}=H$ and for each $i\in\{1,\ldots,p\}$ let $F_i=\{f^i_1,\ldots,f^i_p\}$ be a weak $\ell_i$-march such that $F_i(0)=H$ and $\orb(F_i,v_i)=V(G)$. Then $F_i^{-1}F_i$ is closed and  $\orb(F_i^{-1}F_i,v_i)=V(G)$ for each $i$. Let $F=F_p^{-1}F_p\ldots F_1^{-1}F_1$. Then $\orb(F,v_i)=V(G)$ for each $i\in\{1,\ldots,p\}$ and it follows that $\ccap{1}{G}\geq p$.
\end{proof}

\subsection{A lower bound for Cartesian 1-capacity of a very general collection of graphs}\label{low_Car}

\begin{definition}
Let $P$ be a path of a connected graph $G$ and $H\subseteq V(G)$. We say that $P$ is an {\it occupied path of $H$} if $V(P)\subseteq H$, and $P$ is a {\it vacant path of $H$} if $V(P)\cap H=\emptyset$.
\end{definition}

\begin{proposition}\label{vacant} 
Let $G$ be a connected graph, $a\in H\subseteq V(G)$, $p=|H|$ and let $P=v_1v_2\ldots v_{k}$ be a vacant path of $H$ such that $a$ is adjacent to $v_1$. Then there exists a patient collision-free $k$-march $F=\{f_1,\ldots,f_p\}$ on $G$ such that
	\begin{enumerate}
		\item $H = F(0)$ 
		\item for some $i\in \{1,\ldots,p\}$, it holds that $f_i(0)=a$ and $f_i(k)=v_{k}$
		\item $f_j(t)=f_j(t+1)$ for all $j\in \{1,\ldots,p\}\setminus \{i\}$ and $t \in \{0,1,\ldots,{k-1}\}$
		\item $av_1v_2\ldots v_{{k}-1}$ is a vacant path of $F({k})$.
	\end{enumerate}
\end{proposition}
\begin{proof}
	Let us denote the vertices of $H$ by $\{a,a_2,a_3,\ldots,a_p\}$. We construct a march $F=\{f_1,\ldots,f_p\}$ in the following way:
\begin{align*}
f_1(0)&=a;\\
f_1(t)&=v_{t} \text{ for all } t \in \{1,\ldots,k\};\\
f_i(t)&=a_i, \text{ for all } i \in \{2,\ldots,p\} \text{ and all } t \in \{0,\ldots,k\}.   
\end{align*}   
Clearly, $F(0)=H$, for $i=1$ we have $f_1(0)=a$ and $f_1(k)=v_k$, and $F(k)\cap\{a,v_1,\ldots,v_{k-1}\}=\emptyset$.
\end{proof}

\begin{lemma}\label{shift} Let $G$ be a connected graph, $P=v_0v_1\ldots v_n$ a path in $G$ and $S\subseteq V(G)$ such that $v_0\in S$, $v_n\not \in S$ and $\{v_{\alpha_0},v_{\alpha_1},\ldots,v_{\alpha_k}\}=S\cap V(P)$, where $0=\alpha_0<\alpha_1 <\ldots<\alpha_k<n$. Then there exists a patient collision-free $n$-march $F=\{f_1,\ldots,f_p\}$ on $G$, where $p=|S|$, and $\{s_0,s_1,\ldots,s_k\}\subseteq \{1,\ldots,p\}$ such that 
	\begin{enumerate}
		\item $F(0)=S$
		\item $f_{s_j}(0)=v_{\alpha_j}$, for all $j\in \{0,\ldots,k\}$
		\item  $f_{s_j}(n)=v_{\alpha_{j+1}}$, for all $j\in \{0,\ldots,k-1\}$
		\item $f_{s_k}(n)=v_{n}$
		\item $f_i(0)=f_i(n)$, for all $i\not\in \{s_0,s_1,\ldots,s_k\}$.
	\end{enumerate}
\end{lemma}
\begin{proof} 
Notice that $v_{\alpha_k+1}v_{\alpha_k+2}\ldots v_n$ is a vacant path of $S$. Let $\ell_k=n-\alpha_k$. Then, by Proposition \ref{vacant} there exists a patient collision-free $\ell_k$-march  $F_k=\{f^k_1,\ldots,f^k_p\}$ on $G$ and an $s_k\in\{1,\ldots,p\}$ such that $F_k(0)=S$, $f^k_{s_k}(0)=v_{\alpha_k}$, $f^k_{s_k}(\ell_k)=v_{n}$ and $f_j^k(0)= f_j^k(\ell_k)$ for all $j\in \{1,\ldots,p\}\setminus\{s_k\}$. Let $\{s_0,\ldots,s_k\}\subseteq \{1,\ldots,p\}$ such that $f^k_{s_i}(0)=v_{\alpha_i}$ for all $i\in \{0,\ldots,k\}$.

Continuing inductively, suppose that for some $j$ ($0 < j \leq k$), the sequence of patient collision-free marches $F_k, F_{k-1}, \ldots, F_j$ has been found such that $F_m(0)=F_{m+1}(\ell_{m+1})$ for all $m \in \{j, \ldots, k-1\}$, and $v_{\alpha_{j}}\not \in F_j(\ell_j)$. Then define $f^{j-1}_i(0)=f^{j}_i(\ell_j)$ for each $i\in \{1,\ldots,p\}$, so $f^{j-1}_{s_{j-1}}(0)=v_{\alpha_{j-1}}$. Notice that $v_{\alpha_{j-1}+1}v_{\alpha_{j-1}+2}\ldots v_{\alpha_j}$ is a vacant path of $F_{j-1}(0)$. Let $\ell_{j-1}=\alpha_j-\alpha_{j-1}$. Then, by Proposition \ref{vacant} there exists a patient collision-free $\ell_{j-1}$-march  $F_{j-1}=\{f^{j-1}_1,\ldots,f^{j-1}_p\}$  such that  $f^{j-1}_{s_{j-1}}(0)=v_{\alpha_{j-1}}$, $f^{j-1}_{s_{j-1}}(\ell_{j-1})=v_{\alpha_j}$ and $f_i^{j-1}(0)= f_i^{j-1}(\ell_{j-1})$ for all $i\not=s_{j-1}$.

Then $F := F_0 F_1 \ldots F_k$ has the prescribed properties.
\end{proof}

\begin{figure}[!htbp]
    \centering
\begin{subfigure}{.3\textwidth}
\begin{tikzpicture} [scale=1]
\draw (1,1)--(1,3)--(5,3) (1,2)--(3,2)--(3,3) (3,2)--(4,1);
\draw [blue, thick] (1,1)--(1,2)--(3,2)--(3,3)--(5,3);
\filldraw [redball] (1,1) circle (3pt) node[below] {$\equalto{v_{\alpha_0}}{v_0}$};
\filldraw [whiteball] (4,1) circle (3pt);
\filldraw [whiteball] (1,2) circle (3pt);
\filldraw [greenball] (2,2) circle (3pt) node[below] {$v_{\alpha_1}$};
\filldraw [whiteball] (3,2) circle (3pt);
\filldraw [whiteball] (1,3) circle (3pt);
\filldraw [whiteball] (2,3) circle (3pt);
\filldraw [whiteball] (3,3) circle (3pt);
\filldraw [blueball] (4,3) circle (3pt) node[below] {$v_{\alpha_k}$};
\filldraw [whiteball] (5,3) circle (3pt) node[below] {$v_{n}$};
\end{tikzpicture}
\caption{Initial stage}
\end{subfigure}
\hspace{30 pt}
\begin{subfigure}{.3\textwidth}
\begin{tikzpicture} [scale=1]
\draw (1,1)--(1,3)--(5,3) (1,2)--(3,2)--(3,3) (3,2)--(4,1);
\draw [blue, thick] (1,1)--(1,2)--(3,2)--(3,3)--(5,3);
\filldraw [whiteball] (1,1) circle (3pt) node[below] {$\equalto{v_{\alpha_0}}{v_0}$};
\filldraw [whiteball] (4,1) circle (3pt);
\filldraw [whiteball] (1,2) circle (3pt);
\filldraw [redball] (2,2) circle (3pt) node[below] {$\equalto{v_{\alpha_1}}{f_{s_0}(n)}$};
\filldraw [whiteball] (3,2) circle (3pt);
\filldraw [whiteball] (1,3) circle (3pt);
\filldraw [whiteball] (2,3) circle (3pt);
\filldraw [whiteball] (3,3) circle (3pt);
\filldraw [greenball] (4,3) circle (3pt) node[below=2pt] {$\equalto{v_{\alpha_k}}{f_{s_{k-1}}(n)}$};
\filldraw [blueball] (5,3) circle (3pt) node[below] {$\equalto{v_n}{f_{s_k}(n)}$};
\end{tikzpicture}
\caption{Final stage}
\end{subfigure}
    \caption{Shift from $v_0$ to $v_n$}
    \label{fig:figure_shift}
\end{figure}
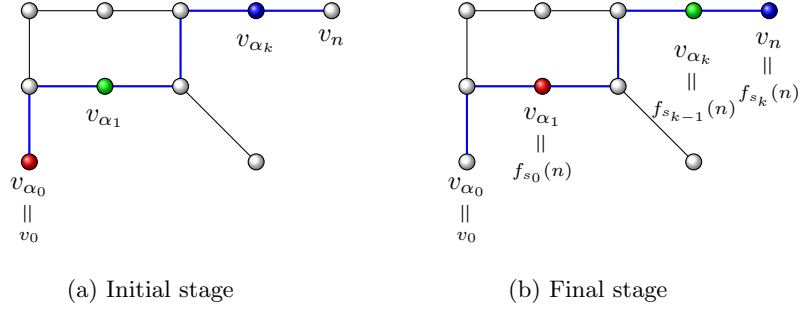

A patient collision-free march $F$ that satisfies the properties of Lemma \ref{shift} is called a {\it shift on $P$} or a {\it shift in the direction from $v_0$ to $v_n$.} Figure \ref{fig:figure_shift} shows a shift in the direction from $v_0$ to $v_n$, where the corresponding path is shown with blue edges.
The next step is to show that occupied vertices can be repositioned in a controlled manner. A shift moves an occupancy along a path while preserving a collision-free patient march. By repeatedly applying such shifts, Proposition \ref{cover} establishes that any configuration of occupied vertices can be transformed into any other configuration of the same size.

\begin{proposition}\label{cover}
	Let $G$ be a connected graph and $H, K\subseteq V(G)$ such that $|H|=|K|=p$. Then, for some $\ell\in \mathbb{N}$ there exists a patient collision-free $\ell$-march $F=\{f_1,\ldots,f_p\}$ such that $F(0)=H$ and $F(\ell)=K$.
\end{proposition}

    \begin{proof}
	 We are going to define $H_t$ to be the occupied vertices of $H\setminus K$ and $K_t$ to be the vacant vertices of $K$ at stage $t$. First let $H_0=H\setminus K$ and $K_0=K\setminus H$. If $H=K$, we are done. Otherwise, pick $y_0\in K_0$ and choose $x_0\in H_0$ such that $\mathrm{d}(x_0,y_0)= \mathrm{d}(H_0,y_0)=:d_0$. Let $P_0=v^0_0v^0_1\ldots v^0_{d_0}$ be a path such that $x_0=v^0_0$ and $v^0_{d_0}=y_0$. Let $\{\alpha^0_j\}_{j=0}^{k_0}$ be an increasing sequence such that  $\{v^0_{\alpha^0_j}\}_{j=0}^{k_0}=H \cap V(P_0)$. These are the occupied vertices on the path. Note $H_0\cap \{v^0_{\alpha^0_j}\}_{j=0}^{k_0}=\{v^0_0\}=\{v^0_{\alpha_0^0}\}$. By letting $S=H$ and applying Lemma \ref{shift}, there exists a patient $\ell_0$-march $F_0=\{f^0_1,\ldots,f^0_p\}$ on $G$ (a shift in the direction from $x_0$ to $y_0$) such that $F_0(0)=H$ and $F_0(\ell_0) = (F_0(0) \setminus \{x_0\}) \cup \{y_0\}.$ Since $x_0 \notin K$ and $y_0 \in K$, it immediately follows that:
            \[F_0(\ell_0) \setminus K = H_0 \setminus \{x_0\} =: H_1\] and
            \[K \setminus F_0(\ell_0) = K_0 \setminus \{y_0\} =: K_1.\]

	Continuing inductively, suppose that $K\supseteq K_0 \supsetneq K_1 \supsetneq \ldots\supsetneq K_m$, $H\supseteq H_0 \supsetneq H_1 \supsetneq \ldots\supsetneq H_m$, and shifts $F_0,\ldots,F_{m-1}$ have been found. If $K_m=\emptyset$ (and hence $H_m=\emptyset$), then $F_{m-1}(\ell_{m-1})=K$ and we are done with the induction. Otherwise, pick $y_m\in K_m$ and choose $x_m\in H_m$ such that $\mathrm{d}(x_m,y_m)= \mathrm{d}(H_m,y_m)=:d_m$. Let $P_m=v^m_0v^m_1\ldots v^m_{d_m}$ be a path such that $x_m=v^m_0$ and $v^m_{d_m}=y_m$. Let $\{\alpha^m_j\}_{j=0}^{k_m}$ be an increasing sequence such that $\{v^m_{\alpha^m_j}\}_{j=0}^{k_m} = F_{m-1}(\ell_{m-1}) \cap V(P_m)$. These are the occupied vertices on the path after the last shift. Note $H_m \cap \{v^m_{\alpha^m_j}\}_{j=0}^{k_m} = \{v^m_0\} = \{v^m_{\alpha^m_0}\}$. By letting $S=F_{m-1}(\ell_{m-1})$ and applying Lemma \ref{shift}, there exists a patient $\ell_m$-march $F_m=\{f^m_1,\ldots,f^m_p\}$ on $G$ (a shift in the direction from $x_m$ to $y_m$) such that $F_m(0)=F_{m-1}(\ell_{m-1})$ and $F_m(\ell_m) = (F_{m-1}(\ell_{m-1}) \setminus \{x_m\}) \cup \{y_m\}$. Since $x_m \notin K$ and $y_m \in K$, it follows that:
    \[F_m(\ell_m) \setminus K = H_m \setminus \{x_m\} =: H_{m+1}\]
    and
    \[K \setminus F_m(\ell_m) = K_m \setminus \{y_m\} =: K_{m+1}.\]
	
	Since $K$ is finite, there exists $n\in{\mathbb N}$ such that $K_n=\emptyset$. Thus $F_{n-1}(\ell_{n-1})=K$. Then it follows that $F:=F_{n-1}F_{n-2}\ldots F_0$ has the prescribed properties.
	\end{proof}

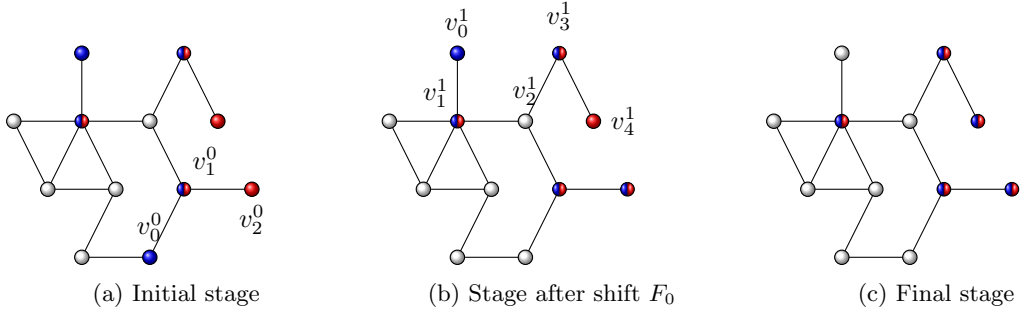
\begin{figure}[htbp]
\centering
\begin{subfigure}{.3\textwidth}
\begin{tikzpicture} [scale=0.9]
\draw (1.5,2)--(2,3)--(2.5,2) (1,3)--(2,3)--(2,4) (2,3)--(3,3)--(3.5,4)--(4,3) (3,3)--(3.5,2)--(4.5,2) (2,1)--(3,1) (3,1)--(3.5,2) (1,3)--(1.5,2)--(2.5,2)--(2,1);
\coordinate (x) at (3.5,2);
 \fill[blueball] (x) + (0, 0.1) arc (90:270:0.1);
 \fill[redball] (x) + (0, -0.1) arc (270:450:0.1) node[above right] {$v^0_1$};
\coordinate (x) at (3.5,4);
 \fill[blueball] (x) + (0, 0.1) arc (90:270:0.1);
 \fill[redball] (x) + (0, -0.1) arc (270:450:0.1) ;
\coordinate (x) at (2,3);
 \fill[blueball] (x) + (0, 0.1) arc (90:270:0.1);
 \fill[redball] (x) + (0, -0.1) arc (270:450:0.1);
\filldraw [whiteball] (1,3) circle (3pt);
\filldraw [whiteball] (3,3) circle (3pt);
\filldraw [redball] (4,3) circle (3pt);
\filldraw [whiteball] (2,1) circle (3pt);
\filldraw [blueball] (3,1) circle (3pt) node[above]{$v^0_0$};
\filldraw [whiteball] (1.5,2) circle (3pt);
\filldraw [whiteball] (2.5,2) circle (3pt);
\filldraw [redball] (4.5,2) circle (3pt) node[below] {$v^0_2$};
\filldraw [blueball] (2,4) circle (3pt);
\end{tikzpicture}
\caption{Initial stage}
\end{subfigure}%
\hspace{10pt}
\begin{subfigure}{.3\textwidth}
\begin{tikzpicture} [scale=0.9]
\draw (1.5,2)--(2,3)--(2.5,2) (1,3)--(2,3)--(2,4) (2,3)--(3,3)--(3.5,4)--(4,3) (3,3)--(3.5,2)--(4.5,2) (2,1)--(3,1) (3,1)--(3.5,2)  (1,3)--(1.5,2)--(2.5,2)--(2,1);
\coordinate (x) at (3.5,2);
 \fill[blueball] (x) + (0, 0.1) arc (90:270:0.1);
 \fill[redball] (x) + (0, -0.1) arc (270:450:0.1);
\coordinate (x) at (3.5,4);
 \fill[blueball] (x) + (0, 0.1) arc (90:270:0.1);
 \fill[redball] (x) + (0, -0.1) arc (270:450:0.1) node[above] {$v^1_3$};
\coordinate (x) at (2,3);
 \fill[blueball] (x) + (0, 0.1) arc (90:270:0.1);
 \fill[redball] (x) + (0, -0.1) arc (270:450:0.1) node[above left] {$v^1_1$};
\coordinate (x) at (4.5,2);
 \fill[blueball] (x) + (0, 0.1) arc (90:270:0.1);
 \fill[redball] (x) + (0, -0.1) arc (270:450:0.1); 
\filldraw [whiteball] (1,3) circle (3pt);
\filldraw [whiteball] (3,3) circle (3pt) node[above] {$v^1_2$};
\filldraw [redball] (4,3) circle (3pt) node[right] {${v^1_4}$};
\filldraw [whiteball] (2,1) circle (3pt);
\filldraw [whiteball] (3,1) circle (3pt);
\filldraw [whiteball] (1.5,2) circle (3pt);
\filldraw [whiteball] (2.5,2) circle (3pt);
\filldraw [blueball] (2,4) circle (3pt) node[above] {$v^1_0$};
\end{tikzpicture}   
\caption{Stage after shift $F_0$}
\end{subfigure}
\hspace{10pt}
\begin{subfigure}{.3\textwidth}
\begin{tikzpicture} [scale=0.9]
\draw (1.5,2)--(2,3)--(2.5,2) (1,3)--(2,3)--(2,4) (2,3)--(3,3)--(3.5,4)--(4,3) (3,3)--(3.5,2)--(4.5,2) (2,1)--(3,1) (3,1)--(3.5,2)  (1,3)--(1.5,2)--(2.5,2)--(2,1);
\coordinate (x) at (3.5,2);
 \fill[blueball] (x) + (0, 0.1) arc (90:270:0.1);
 \fill[redball] (x) + (0, -0.1) arc (270:450:0.1);
\coordinate (x) at (3.5,4);
 \fill[blueball] (x) + (0, 0.1) arc (90:270:0.1);
 \fill[redball] (x) + (0, -0.1) arc (270:450:0.1) ;
\coordinate (x) at (2,3);
 \fill[blueball] (x) + (0, 0.1) arc (90:270:0.1);
 \fill[redball] (x) + (0, -0.1) arc (270:450:0.1);
\coordinate (x) at (4.5,2);
 \fill[blueball] (x) + (0, 0.1) arc (90:270:0.1);
 \fill[redball] (x) + (0, -0.1) arc (270:450:0.1); 
\coordinate (x) at (4,3);
 \fill[blueball] (x) + (0, 0.1) arc (90:270:0.1);
 \fill[redball] (x) + (0, -0.1) arc (270:450:0.1); 
\filldraw [whiteball] (1,3) circle (3pt);
\filldraw [whiteball] (3,3) circle (3pt);
\filldraw [whiteball] (2,1) circle (3pt);
\filldraw [whiteball] (3,1) circle (3pt);
\filldraw [whiteball] (1.5,2) circle (3pt);
\filldraw [whiteball] (2.5,2) circle (3pt);
\filldraw [whiteball] (2,4) circle (3pt);
\end{tikzpicture}
\caption{Final stage}
\end{subfigure}
    \caption{March from $H$ (blue vertices) to $K$ (red vertices)}
    \label{fig:figure-march-H-K}
\end{figure}

Figure \ref{fig:figure-march-H-K} shows three stages of a march from $H$ (blue vertices) to $K$ (red vertices). Vertices that belong to both sets are coloured with both colours.

\begin{proposition}\label{subcover}
    Let $G$ be a connected graph and $H, K\subseteq V(G)$ such that $|K|\leq|H|=p$. Then there exists a patient collision-free ${\ell}$-march $F=\{f_1,\ldots,f_p\}$ on $G$ such that $F(0)=H$ and $F(\ell)\supseteq K$.
\end{proposition}
\begin{proof}
    Let $K'\supseteq K$ such that $|K'|=|H|$ and apply Proposition \ref{cover}.
\end{proof}
	
\begin{lemma}\label{dualpath}
	Let $G$ be a connected graph. Suppose the following
	\begin{enumerate}
		\item $x,v\in V(G)$ and $K\subseteq V(G)$ such that $v$ and $K$ are contained in the same component of $G-x$,
		\item $q=|K|$ and $p=|V(G)|-q$, and
		\item $xv_1v_2\ldots v_{n-1}v$ is a path in $G$.
	\end{enumerate}
	Then there exists a patient collision-free $\ell$-march $F=\{f_1,\ldots,f_p\}$ on $G$ with the dual $D_F=\{g_1,\ldots,g_q\}$ and $s\in \{1,\ldots,p\}$ such that 
	\begin{enumerate}
		\item $f_s(0)=x=f_s(\ell)$
 		\item $D_F(0)=K$
		\item $D_F(\ell)=\{v_1,\ldots,v_q\}$ if $q<n$
		\item $\{v_1,\ldots,v_{n-1},v\}\subseteq D_F(\ell)$ if $q\geq n$.
		\end{enumerate}
\end{lemma}
	\begin{proof}
		Let $m=\min\{n,q\}$.  Let $G_v$ be the component of $G-x$ that contains $v$ (and hence $K$) and  $j:=|V(G_v)|$. Note that the path $v_1\ldots v_{n-1}v$ is contained in $G_v$. Then by Proposition \ref{subcover}, there exists a patient collision-free $\ell$-march $D_F=\{g_1,\ldots,g_q\}$ on $G_v$ such that $D_F(0)=K$ and $D_F(\ell)\supseteq \{v_1,v_2,\ldots,v_m\}$. Let $F_H=\{h_1,\ldots,h_{j-q}\}$ be the dual of $D_F$. Let $\{x_1,\ldots,x_{p+q-j}\}=V(G)\setminus V(G_v)$. 

        Now, define the patient collision-free $\ell$-march $F=\{f_1,\ldots,f_p\}$ on $G$ for all $t\in \{0,\ldots,\ell\}$ by:
            \[
            f_k(t) = \begin{cases} 
            h_k(t), & \text{if } k \in \{1,\ldots,j-q\}, \\
            x_{k+q-j}, & \text{if } k \in \{j-q+1,\ldots,p\}.
            \end{cases}
            \]
        Note that $D_F$ is the dual of $F$ on $G$ and $x=x_c$ for some $c \in \{1, \ldots, p+q-j\}$. So by letting $s = c+j-q$, we have $f_s(0)=x_c=x=f_s(\ell)$. Hence, $F$ and $D_F$ have the prescribed properties.
	\end{proof}

In the study of graph structures, one often encounters arbitrarily long chains of degree-2 vertices, which act as unbranched threads connecting different regions of a network. To characterize graphs where the length of these unbranched paths is strictly bounded, we introduce a parameter that measures how tightly packed the vertices with a degree of at least 3 are. Intuitively, a graph exhibits dense branching if its high-degree vertices appear frequently enough that any sufficiently long path is forced to interact with them---either by passing through an internal branching vertex, or by tightly coupling a branching vertex to a leaf. While this property can be evaluated for any graph, it becomes particularly useful for controlling the structural limits of trees. We formalize this concept with the following definition:

\begin{definition}\label{def:N-dense}
We say that a graph $G$ is {\it branch-vertex $N$-dense} if $N \geq 2$ is the minimum integer such that $G$ contains at least one path of length $N$, and for every path of length $N$, say $P=v_0v_1v_2\ldots v_N$, at least one of the following is true:
\begin{enumerate}
	\item  There is an $i\in\{2,\ldots,N-1\}$  such that $\deg(v_i)\geq 3$
	\item $\deg(v_N)=1$ and $\deg(v_1)\geq 3$
\end{enumerate} 
\end{definition}

\begin{figure}[htbp]
    \centering
\begin{tikzpicture} [scale=1]
\draw (1,3)--(7,3) (4,1)--(1,1)--(2,2)--(2,3) (2,2)--(3,1) (4,3)--(5,2);
\filldraw [blackball] (1,1) circle (3pt);
\filldraw [blackball] (3,1) circle (3pt);
\filldraw [blackball] (4,1) circle (3pt);
\filldraw [blackball] (2,2) circle (3pt);
\filldraw [blackball] (5,2) circle (3pt) node[below] {$v_4$};
\filldraw [blackball] (1,3) circle (3pt) node[above] {$v_0$};
\filldraw [blackball] (2,3) circle (3pt) node[above] {$v_1$};
\filldraw [blackball] (3,3) circle (3pt) node[above] {$v_2$};
\filldraw [blackball] (4,3) circle (3pt) node[above] {$v_3$};
\filldraw [blackball] (5,3) circle (3pt) node[above] {$v_5$};
\filldraw [blackball] (6,3) circle (3pt) node[above] {$v_6$};
\filldraw [blackball] (7,3) circle (3pt) node[above] {$v_7$};
\end{tikzpicture}
    \caption{An example of a branch-vertex 4-dense graph.}
    \label{fig:figure_3}
\end{figure}
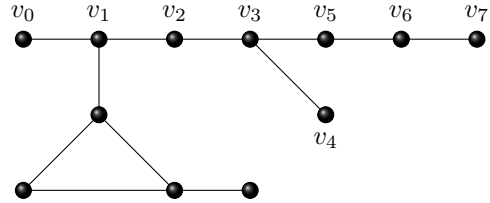

Note that if $G$ is a branch-vertex $N$-dense graph, then $G$ has at least $N+1$ vertices.
Also, any graph with a maximum degree of $2$ (such as paths or cycles) cannot be branch-vertex $N$-dense for any $N$. This is because both conditions of the definition require the existence of a vertex with a degree of at least 3. Conversely, trees with a regular branching structure have a finite density. For instance, one can verify that a complete binary tree of height at least 2 is branch-vertex 4-dense. For an integer $k \geq 1$, the \textit{star graph}, denoted by $K_{1,k}$, is a complete bipartite graph with partition sets of size $1$ and $k$. Note, that if $G$ is a branch-vertex $2$-dense then $G$ is a star graph with at least 3 leaves. 

\begin{lemma}\label{mainstep}
	Suppose that $N\geq 3$ and  $G$ is a connected, branch-vertex $N$-dense graph with $n$ vertices.  Let  $x,v\in V(G)$ be such that $\mathrm{d}(x,v)>N$ and let $K\subseteq V(G)$ be  such that $K$ and $v$ are in the same component of $G-x$ and $|K|=N$. Then there exists a patient collision-free $\ell$-march $F=\{f_1,\ldots,f_{n-N}\}$ on $G$, with the dual $D_F=\{g_1,\ldots,g_N\}$ such that $f_s(0)=x$, $\mathrm{d}(f_s(\ell),v)<\mathrm{d}(x,v)$ and $D_F(\ell)$ is in the same component of $G-f_s(\ell)$ as $v$ for some $s\in \{1,\ldots,n-N\}$.
\end{lemma} 
\begin{proof} 
    Let $v_0v_1v_2\ldots v_{m-1}v_m$ be a shortest path $P$ from $x=v_0$ to $v=v_m$, by assumption $m=\mathrm{d}(x,v)>N$. We apply the branch-vertex $N$-dense property to the subpath $v_0v_1\ldots v_N$. Since $m>N$, the vertex $v_N$ is adjacent to both $v_{N-1}$ and $v_{N+1}$, meaning $\deg(v_N) \ge 2 \neq 1$. Thus, there exists an $i\in\{2,\ldots,N-1\}$ such that $\deg(v_i)\geq 3$. 
	
	Let $y\in V(G)\setminus\{v_{i-1},v_i,v_{i+1}\}$ such that $y$ is adjacent to $v_i$. Since $P$ is a shortest path, $y \notin V(P)$. Then $\mathrm{d}(v_i,v)=m-i$ and hence, 
	\[\mathrm{d}(y,v)\leq \mathrm{d}(v_i,v)+1=m-i+1 < m = \mathrm{d}(x,v).\] 
	So it follows that $y$ is in the same component of $G-x$ as $v$. 

    By applying Proposition \ref{cover} on sets $K$ and $\{v_1,v_2,\ldots,v_{N-1}\}\cup \{y\}$ to the component of $G-x$ containing $v$, there exists a patient collision-free $\ell_1$-march $H_1=\{h^1_1,\ldots,h^1_{N}\}$ on the component of $G-x$ containing $v$, 
    such that $H_1(0)=K$ and $H_1(\ell_1)=\{v_1,v_2,\ldots,v_{N-1}\}\cup \{y\}$. Let $F_1:=\{f_1^1,\ldots,f^1_{n-N}\}=D_{H_1}$ be the dual of $H_1$. Let $s \in \{1,\ldots,n-N\}$ be the index such that $f^1_s(0)=x$. It is clear that $f^1_s(\ell_1)=f^1_s(0)=x$. Furthermore, $D_{F_1}(0)=K$, and $D_{F_1}(\ell_1)=\{v_1,v_2,\ldots,v_{N-1}\}\cup \{y\}$.
	
	Next, since $x$ is adjacent to $v_1$, and the path $v_1 \ldots v_i y$ is a vacant path of $F_1(\ell_1)$, it follows from Proposition \ref{vacant} that there exists a patient collision-free $\ell_2$-march $F_2=\{f^2_1,\ldots,f^2_{n-N}\}$ and its dual $D_{F_2}$ such that $f^2_s(0)=f^1_s(\ell_1)=x$, $f^2_s(\ell_2)=y$, $D_{F_2}(0)=D_{F_1}(\ell_1)$, and $D_{F_2}(\ell_2)=\{x,v_1,\ldots,v_{N-1}\}$.
	
	Let $F=F_2F_1$ and let $D_F=D_{F_2}D_{F_1}$ be its dual. Notice that $f_s(\ell_1+\ell_2)=y \notin \{x,v_1,\ldots,v_{m-1},v_m\}$. Furthermore, $D_F(\ell_1+\ell_2) = \{x,v_1,\ldots,v_{N-1}\} \subseteq  \{x,v_1,\ldots,v_{m-1},v\}$.  Therefore, $D_F(\ell_1+\ell_2)$ is in the same component of $G-f_s(\ell_1+\ell_2)$ as $v$, and the result follows.
\end{proof}

\begin{lemma}\label{switch}
	Suppose that $N\geq 3$ and $G$ is a connected, branch-vertex $N$-dense graph with $n$ vertices. Then for every $x,v\in V(G)$, there exists a patient collision-free $\ell$-march  $F=\{f_1,\ldots,f_{n-N}\}$ on $G$, such that $f_s(0)=x$ and $f_s(\ell)=v$ for some $s\in \{1,\ldots,n-N\}$.
\end{lemma}
\begin{proof}
	First, if $x=v$ then let $f_s$ be the identity for each $s$ and we are done. So suppose that $x\not=v$. 
    
    Let $H\subseteq V(G)$ be such that $|H|=N$ and $x\not\in H$. Let $K\subseteq V(G)$ such that $v\in K$, $|K|=N$ and $G[K]$ is connected. Set $p:=n-N$.

     Then, by Proposition \ref{cover} there exists a patient collision-free $\ell_0$-march $D_{F_0}=\{g^0_1,\ldots,g^0_{N}\}$ on $G$ such that $D_{F_0}(0)=H$ and $D_{F_0}(\ell_0)=K$. Let
    $F_0=\{f^0_1,\ldots,f^0_{p}\}$ be its dual. Then there exists $s\in\{1,\ldots,p\}$ such that $f^0_s(0)=x\notin H$ and $x_0:=f^0_s(\ell_0)\not\in K$.
	Note that  $D_{F_0}(\ell_0)$ is in the same component of $G-x_0$ as $v$.

    Given $x_0$, we will proceed in finding $x_1,x_2,\ldots,x_j$ inductively in the following manner:
	
	Suppose that for $j\geq 0$, $x_j\in V(G)$ and $D_{F_j}(\ell_j)\subseteq V(G)$ are defined such that $x_j\not \in D_{F_j}(\ell_j)$ and  $D_{F_j}(\ell_j)$ is in the same component of $G-x_j$ as $v$. We have two cases to consider:\\
	
	\noindent{{\bf Case $1(j)$}}: $\mathrm{d}(x_j,v)>N$
	
	Then by Lemma \ref{mainstep}, there is a patient collision-free $\ell_{j+1}$-march $F_{j+1}=\{f^{j+1}_1,\ldots,f^{j+1}_{n-N}\}$ on $G$, with the dual $D_{F_{j+1}}=\{g^{j+1}_1,\ldots,g^{j+1}_N\}$ such that $f^{j+1}_s(0)=x_j$, $\mathrm{d}(f^{j+1}_s(\ell_{j+1}),v) < \mathrm{d}(x_j,v)$ and $D_{F_{j+1}}(\ell_{j+1})$ is in the same component of $G - f^{j+1}_s(\ell_{j+1})$ as $v$. Now let $x_{j+1}:=f^{j+1}_s(\ell_{j+1})$.\\
	
	\noindent{{\bf Case $2(j)$}}: $\mathrm{d}({x_j},v)\leq N$ 

    There are two parts in this case. The first is vacating a path from $v$ to a vertex adjacent to $x_j$. The second is moving the player on $x_j$ to vertex $v$.
	
	Let $m:=\mathrm{d}({x_j},v)$ and let $x_j u_1 u_2 \ldots u_{m-1} v$ be a shortest path from $x_j$ to $v$. By Lemma \ref{dualpath}, since $D_{F_j}(\ell_j)$ is contained in the component of $G-x_j$ that contains $v$, there is a patient collision-free $\ell_{j+1}$-march $F_{j+1}=\{f^{j+1}_1,\ldots,f^{j+1}_p\}$ on $G$ with the dual $D_{F_{j+1}}$ such that $f^{j+1}_s(0)=x_j=f^{j+1}_s(\ell_{j+1})$ and $\{u_1, u_2, \ldots,\allowbreak u_{m-1}, v\} \subseteq D_{F_{j+1}}(\ell_{j+1})$. Hence, $u_1 u_2 \ldots u_{m-1} v$ is a vacant path adjacent to $x_j$. By Proposition \ref{vacant}, there is a patient collision-free $\ell_{j+2}$-march $F_{j+2}=\{f^{j+2}_1,\ldots,f^{j+2}_p\}$ on $G$ such that $f^{j+2}_s(0)=x_j$ and $f^{j+2}_s(\ell_{j+2})=v$.
	
Since $\mathrm{d}(x_0,v)\leq n$ and $\mathrm{d}(x_j,v)< \mathrm{d}(x_{j-1},v)$, there exists a $J\geq 0$ such that $\mathrm{d}(x_J,v)\leq N$. Then $F:=F_{J+2}F_{J+1}\ldots F_0$ has the prescribed properties.
\end{proof}

\begin{figure}[htb]
    \centering
\begin{subfigure}{.4\textwidth}
\begin{tikzpicture} [scale=0.7]
\draw (1,2)--(7,2) (3.5,3)--(2.5,3)--(2,2)--(3,2)--(2.5,3) (3.5,1)--(4.5,1)--(5,2)--(5.5,1)--(4,1);
\filldraw [blackball] (3.5,1) circle (3pt);
\filldraw [whiteball] (4.5,1) circle (3pt);
\filldraw [whiteball] (5.5,1) circle (3pt) node[right] {$f_s(l_1)$};
\filldraw [blackball] (1,2) circle (3pt);
\filldraw [blackball] (2,2) circle (3pt);
\filldraw [redball] (3,2) circle (3pt) node[below] {$x$};
\filldraw [blackball] (4,2) circle (3pt);
\filldraw [whiteball] (5,2) circle (3pt);
\filldraw [blackball] (6,2) circle (3pt);
\filldraw [blackball] (7,2) circle (3pt) node[below] {$v$};
\filldraw [blackball] (2.5,3) circle (3pt);
\filldraw [blackball] (3.5,3) circle (3pt);
\end{tikzpicture}
\caption{Initial position}
\end{subfigure}
\hspace{20pt}
\begin{subfigure}{.4\textwidth}
\begin{tikzpicture} [scale=0.7]
\draw (1,2)--(7,2) (3.5,3)--(2.5,3)--(2,2)--(3,2)--(2.5,3) (3.5,1)--(4.5,1)--(5,2)--(5.5,1)--(4,1);
\filldraw [blackball] (3.5,1) circle (3pt);
\filldraw [blackball] (4.5,1) circle (3pt);
\filldraw [whiteball] (5.5,1) circle (3pt) node[right] {$f_s(l_1)$};
\filldraw [blackball] (1,2) circle (3pt);
\filldraw [blackball] (2,2) circle (3pt);
\filldraw [redball] (3,2) circle (3pt) node[below] {$x$};
\filldraw [whiteball] (4,2) circle (3pt);
\filldraw [whiteball] (5,2) circle (3pt);
\filldraw [blackball] (6,2) circle (3pt);
\filldraw [blackball] (7,2) circle (3pt) node[below] {$v$};
\filldraw [blackball] (2.5,3) circle (3pt);
\filldraw [blackball] (3.5,3) circle (3pt);
\end{tikzpicture}
\caption{Vacant vertices aligned so red player can move to vertex $f_s(\ell_1)$}
\end{subfigure}
\hspace{20pt}
\begin{subfigure}{.4\textwidth}
\begin{tikzpicture} [scale=0.7]
\draw (1,2)--(7,2) (3.5,3)--(2.5,3)--(2,2)--(3,2)--(2.5,3) (3.5,1)--(4.5,1)--(5,2)--(5.5,1)--(4,1);
\filldraw [blackball] (3.5,1) circle (3pt);
\filldraw [blackball] (4.5,1) circle (3pt);
\filldraw [redball] (5.5,1) circle (3pt) node[right] {$f_s(l_1)$};
\filldraw [blackball] (1,2) circle (3pt);
\filldraw [blackball] (2,2) circle (3pt);
\filldraw [whiteball] (3,2) circle (3pt) node[below] {$x$};
\filldraw [whiteball] (4,2) circle (3pt);
\filldraw [whiteball] (5,2) circle (3pt);
\filldraw [blackball] (6,2) circle (3pt);
\filldraw [blackball] (7,2) circle (3pt) node[below] {$v$};
\filldraw [blackball] (2.5,3) circle (3pt);
\filldraw [blackball] (3.5,3) circle (3pt);
\end{tikzpicture}
\caption{Red player moves to $f_s(\ell_1)$}
\end{subfigure}
\hspace{20pt}
\begin{subfigure}{.4\textwidth}
\begin{tikzpicture} [scale=0.7]
\draw (1,2)--(7,2) (3.5,3)--(2.5,3)--(2,2)--(3,2)--(2.5,3) (3.5,1)--(4.5,1)--(5,2)--(5.5,1)--(4,1);
\filldraw [blackball] (3.5,1) circle (3pt);
\filldraw [blackball] (4.5,1) circle (3pt);
\filldraw [redball] (5.5,1) circle (3pt) node[right] {$f_s(l_1)$};
\filldraw [blackball] (1,2) circle (3pt);
\filldraw [blackball] (2,2) circle (3pt);
\filldraw [blackball] (3,2) circle (3pt) node[below] {$x$};
\filldraw [blackball] (4,2) circle (3pt);
\filldraw [whiteball] (5,2) circle (3pt);
\filldraw [whiteball] (6,2) circle (3pt);
\filldraw [whiteball] (7,2) circle (3pt) node[below] {$v$};
\filldraw [blackball] (2.5,3) circle (3pt);
\filldraw [blackball] (3.5,3) circle (3pt);
\end{tikzpicture}
\caption{Vacant vertices aligned so red player can move to vertex $v$}
\end{subfigure}
\hspace{20pt}
\begin{subfigure}{.4\textwidth}
\begin{tikzpicture} [scale=0.7]
\draw (1,2)--(7,2) (3.5,3)--(2.5,3)--(2,2)--(3,2)--(2.5,3) (3.5,1)--(4.5,1)--(5,2)--(5.5,1)--(4,1);
\filldraw [blackball] (3.5,1) circle (3pt);
\filldraw [blackball] (4.5,1) circle (3pt);
\filldraw [whiteball] (5.5,1) circle (3pt) node[right] {$f_s(l_1)$};
\filldraw [blackball] (1,2) circle (3pt);
\filldraw [blackball] (2,2) circle (3pt);
\filldraw [blackball] (3,2) circle (3pt) node[below] {$x$};
\filldraw [blackball] (4,2) circle (3pt);
\filldraw [whiteball] (5,2) circle (3pt);
\filldraw [whiteball] (6,2) circle (3pt);
\filldraw [redball] (7,2) circle (3pt) node[below] {$v$};
\filldraw [blackball] (2.5,3) circle (3pt);
\filldraw [blackball] (3.5,3) circle (3pt);
\end{tikzpicture}
\caption{Final position}
\end{subfigure}
    \caption{Movement from the proof of Lemma \ref{switch}}
    \label{fig:figure_4}
\end{figure}
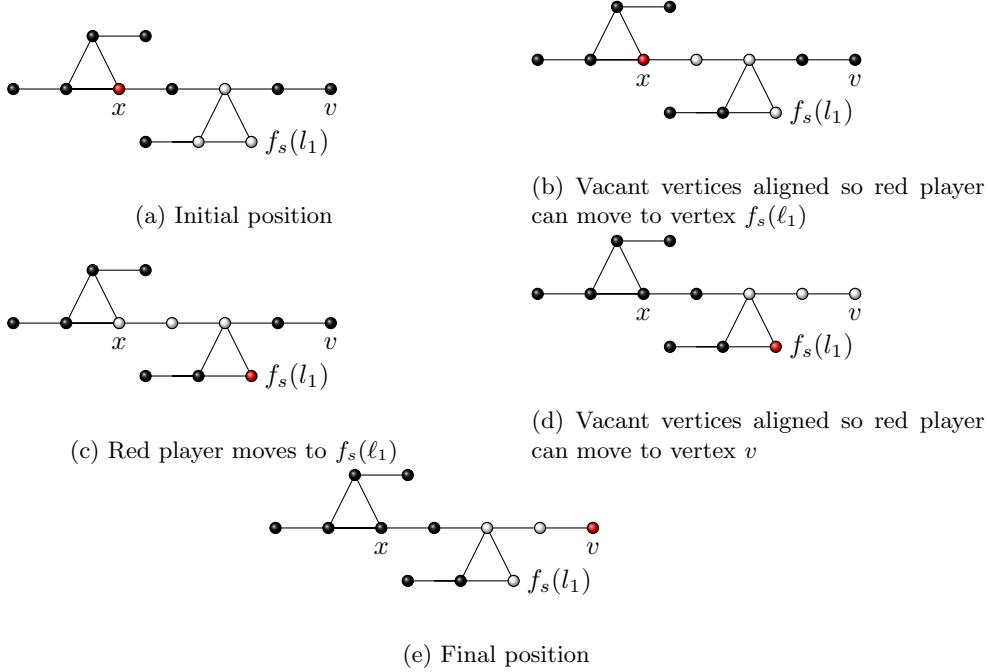
Movement from the proof of Lemma \ref{switch} is illustrated in Figure \ref{fig:figure_4}.
Moving from left to right and top to bottom, the vacant vertices (white) are moved so that the player in red at vertex $x$ can move to vertex $f_s(\ell-1)$. Then the vacant vertices are realigned so the player in red can then move from vertex $f_s(\ell-1)$ to the target vertex $v$. 

 \begin{theorem}\label{almost}
 	Suppose that $N\geq 3$ and $G$ is a connected, branch-vertex $N$-dense graph with $n$ vertices. Then for every $v\in V(G)$, there exists a patient collision-free $\ell$-march $F=\{f_1,\ldots,f_{n-N}\}$ on $G$,  such that $\orb(F,v)=V(G)$.
 \end{theorem}
 \begin{proof}
 Let $v\in V(G)$ be arbitrary and denote the vertices of $G$ by $V(G):=\{v_1,\ldots,v_n\}$, such that $v=v_1$. From  Lemma \ref{switch} we have that for each $i\in \{1,\ldots,n-1\}$, there exists a patient collision-free $\ell_i$-march $F^{i}=\{f^{i}_1,\ldots,f^{i}_{n-N}\}$ on $G$ such that for some $s_i\in\{1,\ldots,n-N\}$ we have
 $v_i=f^i_{s_i}(0)$ and $v_{i+1}=f^i_{s_i}(\ell_i)$. Let $F=F^{n-1}F^{n-2}\ldots F^1$. Then $\orb(F,v)=V(G)$.
 \end{proof}

\begin{proposition}\label{full}
	Suppose there is a non-empty $C \subseteq V(G)$ of order $p$ such that for every $v\in C$, there exists a patient collision-free march $F^v=\{f^v_1,\ldots,f^v_p\}$ on $G$ with $F^v(0)=C$ and $\orb(F^v,v)=V(G)$. Then there exists a patient collision-free tour $F=\{f_1,\ldots,f_p\}$ on $G$ with $F(0)=C$. 
\end{proposition}
\begin{proof} 
	Let $C = \{v_1, \ldots, v_p\}$. For each $i \in \{1, \ldots, p\}$, let $F_i$ be a patient collision-free march such that $F_i(0)=C$ and $\orb(F_i,v_i)=V(G)$. We define the march $F$ as $F = F_p^{-1}F_p \ldots F_2^{-1}F_2 F_1^{-1}F_1.$ Then, $\orb(F,v_i)=V(G)$ for all $v_i \in C$, hence $F$ is a tour on $G$.
\end{proof}
 
 \begin{corollary}\label{lowerbound}
 If $N\geq 3$ is an integer and $G$ is a connected branch-vertex $N$-dense graph with $n$ vertices, then $\ccap{1}{G}\geq n-N$.
 \end{corollary}
 \begin{proof}
 	Follows directly from Theorem \ref{almost} and Propositions \ref{together}, \ref{full}.
 \end{proof}

\subsection{An upper bound for Cartesian 1-capacity of a general collection of graphs}\label{upp_Car}

Let $ab$ be an edge of connected graph $G$. We say that $ab$ is a {\it bridge} of $G$ if 
$ab$ is a cut-edge. Notice that if $ab$ is a bridge of the graph $G$, then $G-ab$ has exactly two components. 

\begin{definition}\label{def:m-bridge}
	Let $G$ be a connected graph and let $M\geq 1$ be an integer. We say that path $B_M=v_0v_1\ldots v_M$ is an {\it $M$-bridge} of $G$ if 
	\begin{enumerate}
		\item $v_iv_{i+1}$ is a bridge of $G$ for each $i\in\{0,\ldots,M-1\}$,
		\item $\deg(v_i)=2$ for each $i\in\{1,\ldots,M-1\}$, and
		\item $\deg(v_i)\geq 2$ for each $i\in\{0,M\}$.
	\end{enumerate}
\end{definition}

For example in Figure \ref{fig:figure_3} the paths $v_1v_2v_3$ and $v_3v_5v_6$ are 2-bridges while the path $v_3v_5v_6v_7$ is not a 3-bridge.

Note that if $v_0v_1\ldots v_M$ is an {\it $M$-bridge} of $G$, for some $M\geq2$, then $v_0v_1\ldots v_{M-1}$ is an {\it $(M-1)$-bridge} of $G$, so it makes sense to determine the largest $M$ such that $G$ contains an $M$-bridge.

\begin{definition}{}{}
Let $G$ be a connected graph. We define $\brig(G)$ as follows:
\[
\brig(G) = 
\begin{dcases} 
        0, & \text{if $G$ has no $1$-bridge}, \\
      \max\{M \mid G \text{ contains an } M\text{-bridge}\}, & \text{otherwise.} \\
\end{dcases}
\]
\end{definition}

Let $B_M=v_0v_1\ldots v_M$ be an $M$-bridge of a graph $G$. If $M=1$, we denote by $G-B_M$ the graph $G-v_0v_1$. For $M>1$ we denote by $G-B_M$ the graph $G-\{v_1,\ldots,v_{M-1}\}$.
Notice that if $G$ is connected, then $G-B_M$ has exactly two components $C_{v_0}$ and $C_{v_M}$, where $v_0\in C_{v_0}$ and $v_M\in C_{v_M}$.

\begin{definition}\label{def:levels}
Let $B_M=v_0v_1\ldots v_M$ be an $M$-bridge of a connected graph $G$ and $C_{v_0}$ the component of $G-B_M$ such that $v_0\in C_{v_0}$. Then 
for each $v\in V(G)$ we define the {\it level of $v$ with respect to} $v_0$ by
$$\lev_{v_0}(v) =
\begin{cases}
    -\mathrm{d}(v, v_0), & \text{if } v\in C_{v_0}, \\
    \mathrm{d}(v,v_0), & \text{if } v\not\in C_{v_0}.
\end{cases}
$$
\end{definition}

Note that $\lev_{v_0}(v_0)=0$, $\lev_{v_0}(v)\geq M$, if $v\in C_{v_M}$, and $\lev_{v_0}(v)\leq 0$, if $v\in C_{v_0}$. 

Let $G$ be a connected graph on $n$ vertices and $B_M=v_0v_1\ldots v_M$ an $M$-bridge of $G$. For what follows, suppose that $F=\{f_1,\ldots,f_{n-M}\}$ is a patient collision-free $\ell$-march on $G$ with the dual $D_{F}=\{g_1,\ldots,g_{M}\}$. Notice that if for some $J\in \{1,\ldots, n-M\}$ and some $t\in\{0,1,\ldots,\ell-1\}$ it holds that $\lev_{v_0}(f_J(t+1))=\lev_{v_0}(f_J(t))+1$, then there exists $K\in \{1,\ldots,M\}$ such that $f_J(t+1)=g_K(t)$ and $g_K(t+1)=f_J(t)$. Hence, $\lev_{v_0}(g_K(t+1))=\lev_{v_0}(g_K(t))-1$. Furthermore, for any $k\in\{0,1,\ldots,M\}$, $\lev_{v_0}(v)=k$ if and only if $v=v_k\in B_M$.

For any $v\in V(G)$ and any $t\in\{0,1,\ldots,\ell\}$ we define 
\begin{enumerate}
	\item $W^{v_0}_<(t,v)=\{j \mid g_j(t)\in D_F(t) \text{ and } \lev_{v_0}(g_j(t))<\lev_{v_0}(v)\}$,
	\item $W^{v_0}_=(t,v)=\{j \mid g_j(t)\in D_F(t) \text{ and } \lev_{v_0}(g_j(t))=\lev_{v_0}(v)\}$, and
	\item $W^{v_0}_>(t,v)=\{j \mid g_j(t)\in D_F(t) \text{ and } \lev_{v_0}(g_j(t))>\lev_{v_0}(v)\}$.
\end{enumerate}
Note that  $W^{v_0}_<(t,v)$,  $W^{v_0}_=(t,v)$ and  $W^{v_0}_>(t,v)$ are pairwise disjoint.

In the context of players moving on the graph, let $D_F(t)$ represent the set of vacant vertices in the graph at time $t$, and let $g_j(t)$ denote the position of the $j$-th vacancy. Given a reference vertex $v \in V(G)$, these three sets partition the indices of the vacancies based strictly on their level with respect to an endpoint $v_0$ of a given $M$-bridge and relative to $v$:
\begin{itemize}
	\item $W^{v_0}_<(t,v)$ is the index set of all vacancies at time $t$ that are located at a strictly lower level than $v$.
	\item $W^{v_0}_=(t,v)$ is the index set of all vacancies at time $t$ that are located at the exact same level as $v$.
	\item $W^{v_0}_>(t,v)$ is the index set of all vacancies at time $t$ that are located at a strictly higher level than $v$.
\end{itemize}
An example is given in Figure \ref{fig:figure_6}. The white vertices are in $D_F(t)$, for some stage $t$, and for vertex $v_1$ we have $|W^{v_0}_<(t,v_1)|=2$, $|W^{v_0}_=(t,v_1)|=1$ and $|W^{v_0}_>(t,v_1)|=3$.

\begin{figure}[htbp]
    \centering
\begin{tikzpicture} [scale=1]
\draw (1,2)--(2,3)--(3,2)--(4,2)--(5,2)--(6,3)--(7,3)--(8,3) (1,2)--(2,1)--(3,2) (5,2)--(6,2)--(7,2)--(7,3)--(8,2) (5,2)--(6,1);
\draw[dotted] (1,2)--(1,0.6);
\draw[dotted] (2,3)--(2,0.6);
\draw[dotted] (3,2)--(3,0.6);
\draw[dotted] (4,2)--(4,0.6);
\draw[dotted] (5,2)--(5,0.6);
\draw[dotted] (6,3)--(6,0.6);
\draw[dotted] (7,2)--(7,0.6);
\draw[dotted] (8,3)--(8,0.6);
\filldraw [whiteball] (2,1) circle (3pt) node[below] {};
\filldraw [whiteball] (6,1) circle (3pt) node[below] {};
\filldraw [whiteball] (1,2) circle (3pt) node[below] {};
\filldraw [blackball] (3,2) circle (3pt) node[below] {$v_0$};
\filldraw [whiteball] (4,2) circle (3pt) node[below] {$v_1$};
\filldraw [blackball] (5,2) circle (3pt) node[below] {$v_2$};
\filldraw [blackball] (6,2) circle (3pt);
\filldraw [whiteball] (7,2) circle (3pt) node[below] {};
\filldraw [blackball] (8,2) circle (3pt) node[below] {};
\filldraw [blackball] (2,3) circle (3pt);
\filldraw [whiteball] (6,3) circle (3pt);
\filldraw [blackball] (7,3) circle (3pt);
\filldraw [blackball] (8,3) circle (3pt);
\node at (0,0.4) {levels:};
\node at (1,0.4) {$-2$};
\node at (2,0.4) {$-1$};
\node at (3,0.4) {$0$};
\node at (4,0.4) {$1$};
\node at (5,0.4) {$2$};
\node at (6,0.4) {$3$};
\node at (7,0.4) {$4$};
\node at (8,0.4) {$5$};
\end{tikzpicture}
    \caption{The levels with respect to $v_0$. }
    \label{fig:figure_6}
\end{figure}
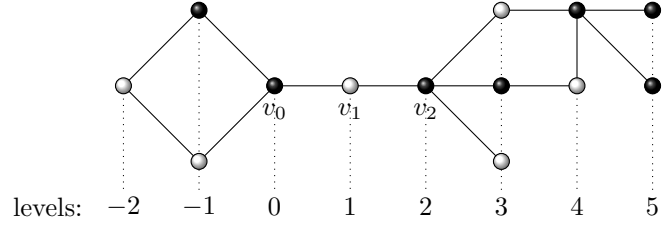

To establish an upper bound for the Cartesian $1$-capacity, we analyse the behaviour of players and vacancies along an $M$-bridge. The key idea is to track how the distribution of vacant vertices changes when a player moves between consecutive levels of the bridge. Propositions \ref{bridge}--\ref{trap} formalize this behaviour and show that, when the number of vacant vertices is sufficiently small, a player cannot traverse the entire bridge. These results form the foundation for the main upper bound established in Theorem \ref{bridgebound} and Corollary \ref{bridgecor}.

\begin{proposition}\label{bridge}
	Let $B_M=v_0v_1\ldots v_M$ be an $M$-bridge of a connected graph $G$ on $n$ vertices and $F=\{f_1,\ldots,f_p\}$ a patient collision-free $\ell$-march on $G$. If for some $t\in\{0, \ldots, \ell-1\}$ and some $J\in\{1,\ldots,p\}$ it holds that $\lev_{v_0}(f_J(t+1))= \lev_{v_0}(f_J(t))+1$, then 
	\[|W^{v_0}_<(t,f_J(t))|+1 \leq |W^{v_0}_<(t+1,f_J(t+1))| \text{ and } |W^{v_0}_>(t+1,f_J(t+1))|\leq |W^{v_0}_>(t,f_J(t))|-1.\] 
	Similarly, if $\lev_{v_0}(f_J(t+1))= \lev_{v_0}(f_J(t))-1$, then 
	\[|W^{v_0}_>(t,f_J(t))|+1\leq |W^{v_0}_>(t+1,f_J(t+1))| \text{ and } |W^{v_0}_<(t+1,f_J(t+1))|\leq |W^{v_0}_<(t,f_J(t))|-1.\]
\end{proposition}
\begin{proof}
    Let $D_F=\{g_1, \ldots, g_{n-p}\}$ be the dual of $F$. 
    
    Suppose that $\lev_{v_0}(f_J(t+1))= \lev_{v_0}(f_J(t))+1$ and $r\in\{1,\ldots,n-p\}$ is such that $g_r(t)=f_J(t+1)$ and $g_r(t+1)=f_J(t)$. Then $r\in W^{v_0}_>(t,f_J(t))$ and $r\in W^{v_0}_<(t+1,f_J(t+1))$. Moreover, since $D_F$ is the dual of a patient march and is therefore also a patient march, for any $k \in \{1,\ldots,n-p\} \setminus \{r\}$, $\lev_{v_0}(g_k(t+1)) = \lev_{v_0}(g_k(t))$. Hence, \[W^{v_0}_<(t+1,f_J(t+1))=W^{v_0}_<(t,f_J(t))\cup W^{v_0}_=(t,f_J(t))\cup \{r\}\] and
	\[W^{v_0}_>(t+1,f_J(t+1)) \subseteq W^{v_0}_>(t,f_J(t)) \setminus \{r\}.\] 
	The desired inequalities follow immediately from the cardinalities of these sets. The proof is analogous for when $\lev_{v_0}(f_J(t+1))= \lev_{v_0}(f_J(t))-1$.
\end{proof}

\begin{proposition}\label{stay}
    Let $B_M=v_0v_1\ldots v_M$ be an $M$-bridge of a connected graph $G$ on $n$ vertices and $F=\{f_1,\ldots,f_p\}$ a patient collision-free $\ell$-march on $G$.
    Suppose that for some $t\in\{0, \ldots, \ell-1\}$ and some $J\in\{1,\ldots,p\}$, $f_J(t)\in V(B_M)$. If $f_J(t)=f_J(t+1)$, then 
	\[W^{v_0}_<(t,f_J(t))=W^{v_0}_<(t+1,f_J(t+1)) \text{ and } W^{v_0}_>(t,f_J(t))=W^{v_0}_>(t+1,f_J(t+1)).\]
\end{proposition}
\begin{proof}
    Let $D_F=\{g_1, \ldots, g_{n-p}\}$ be the dual of $F$. First, note that since $f_J(t)\in V(B_M)$, the only vertex in $G$ at level $\lev_{v_0}(f_J(t))$ is $f_J(t)$ itself. Since $F$ is a collision-free march and $D_F$ its dual, for all $k\in\{1,\ldots,n-p\}$ it holds that $g_k(t)\neq f_J(t)$ and $g_k(t+1)\neq f_J(t+1)=f_J(t)$.
    Hence, \[W^{v_0}_=(t,f_J(t))=W^{v_0}_=(t+1,f_J(t+1))=\emptyset.\] 

    Towards a contradiction, suppose that there exists $r\in W^{v_0}_<(t,f_J(t))$ such that $r\in W^{v_0}_>(t+1,f_J(t+1))$. Then $\lev_{v_0}(g_r(t))\leq \lev_{v_0}(f_J(t))-1= \lev_{v_0}(f_J(t+1))-1$ and $\lev_{v_0}(f_J(t+1))+1\leq \lev_{v_0}(g_r(t+1))$. This implies $\mathrm{d}(g_r(t),g_r(t+1))\geq 2$, a contradiction since $g_r(t)g_r(t+1)$ is an edge of $G$. Hence, $W^{v_0}_<(t,f_J(t))\subseteq W^{v_0}_<(t+1,f_J(t+1))$. 

    The proofs that $W^{v_0}_<(t+1,f_J(t+1))\subseteq W^{v_0}_<(t,f_J(t))$, $W^{v_0}_>(t,f_J(t))\subseteq W^{v_0}_>(t+1,f_J(t+1))$, and $W^{v_0}_>(t+1,f_J(t+1))\subseteq W^{v_0}_>(t,f_J(t))$ are similar. It follows that the sets are equal.
\end{proof}

\begin{proposition}\label{move}
	Let $B_M=v_0v_1\ldots v_M$ be an $M$-bridge of a connected graph $G$ on $n$ vertices and $F=\{f_1,\ldots,f_p\}$ a patient collision-free $\ell$-march on $G$. Suppose $t \in \{0, \ldots, \ell-1\}$ and $J \in \{1, \ldots, p\}$ are arbitrary.
	\begin{enumerate}
        \item If $-1\leq \lev_{v_0}(f_J(t))\leq M-1$ and $\lev_{v_0}(f_J(t+1))=\lev_{v_0}(f_J(t))+1$, then 
        \[|W^{v_0}_>(t+1,f_J(t+1))|=|W^{v_0}_>(t,f_J(t))|-1.\] 
		\item If $0\leq \lev_{v_0}(f_J(t))\leq M$ and $\lev_{v_0}(f_J(t+1))=\lev_{v_0}(f_J(t))-1$, then 
        \[|W^{v_0}_>(t+1,f_J(t+1))|=|W^{v_0}_>(t,f_J(t))|+1.\]
	\end{enumerate}
\end{proposition}
\begin{proof} 
    Let $D_F=\{g_1, \ldots, g_{n-p}\}$ be the dual of $F$. Suppose $-1\leq \lev_{v_0}(f_J(t))\leq M-1$ and $\lev_{v_0}(f_J(t+1))=\lev_{v_0}(f_J(t))+1$. Since $F$ is a patient march, there exists $k\in\{1,\ldots, n-p\}$ such that $f_J(t)=g_k(t+1)$ and $f_J(t+1)=g_k(t)$. Consequently, $k\in W^{v_0}_>(t,f_J(t))$ and $k\in W^{v_0}_<(t+1,f_J(t+1))$. 
    
    Since $\lev_{v_0}(f_J(t+1)) \in \{0, \dots, M\}$, then $f_J(t+1)\in V(B_M)$. Since $F$ is collision-free, it follows that $W^{v_0}_=(t+1,f_J(t+1))=\emptyset$. It follows that $W^{v_0}_>(t+1,f_J(t+1))=W^{v_0}_>(t,f_J(t))\setminus\{r\}$. Hence, $|W^{v_0}_>(t+1,f_J(t+1))|=|W^{v_0}_>(t,f_J(t))|-1$.
	
	The proof for 2. is analogous. 
\end{proof}

\begin{proposition}\label{trap}
    Let $B_M=v_0v_1\ldots v_M$ be an $M$-bridge of a connected graph $G$ on $n$ vertices, and let $F=\{f_1,\ldots,f_p\}$ be a patient collision-free $\ell$-march on $G$, where $p \geq n - (M + 1)$. For any $J \in \{1,\ldots,p\}$ and any $t_0 \in \{0,\ldots,\ell\}$, if $\lev_{v_0}(f_J(t_0))\leq -1$, then $\lev_{v_0}(f_J(t))\leq M$ for all $t \in \{0, \dots, \ell\}$.
\end{proposition}
\begin{proof} 
    Let $q = n - p$. By our hypothesis, $q \leq M+1$. Let $\lev_{v_0}(f_J(t_0))\leq -1$. Suppose, on the contrary, that there exists $m > t_0$ such that $\lev_{v_0}(f_J(m)) = M+1$. Let $t_1, t_2\in\{0,\ldots,\ell\}$ be such that $t_0\leq t_1<t_2\leq m$, where $\lev_{v_0}(f_J(t_1))=-1$, $\lev_{v_0}(f_J(t_2))=M+1$, and $0\leq\lev_{v_0}(f_J(k))\leq M$ for all $t_1<k<t_2$. 

     It follows from Proposition \ref{bridge} that $|W^{v_0}_<(t_1+1,f_J(t_1+1))|\geq 1$. Likewise, it follows from a simple induction using Propositions \ref{stay} and \ref{move} that if $f_J(k)=v_c$, then $|W^{v_0}_<(k,f_J(k))| = |W^{v_0}_<(t_1+1,f_J(t_1+1))|+c \geq c+1$, for each $k\in\{t_1+1,\ldots,t_2-1\}$.

     Because $\lev_{v_0}(f_J(t_2))=M+1$ and the choice of $t_2$, we have $\lev_{v_0}(f_J(t_2-1))=M$. Hence, $|W^{v_0}_<(t_2-1,f_J(t_2-1))|\geq M+1$. 

     Since $f_J(t_2-1)\not=f_J(t_2)$, there exists $r\in \{1,\ldots,q\}$ such that $g_r(t_2-1)=f_J(t_2)$. Therefore, $\lev_{v_0}(g_r(t_2-1))=M+1$ and it follows that $r \not\in W^{v_0}_<(t_2-1,f_J(t_2-1))$.

     Thus, 
     \begin{align*}
         q=|D_F(t_2-1)| &\geq |W^{v_0}_<(t_2-1,f_J(t_2-1))\cup \{r\}| \\
                        & =|W^{v_0}_<(t_2-1,f_J(t_2-1))|+1\\
                        & \geq M+2,
     \end{align*}
     which contradicts our initial assumption that $q \leq M+1$.

    Furthermore, for $m < t_0$ the things can be proven analogously by considering the appropriate set $W^{v_0}_>$, and obtaining the same contradiction. Thus, $\lev_{v_0}(f_J(t)) \leq M$ for all $t \in \{0, \dots, \ell\}$.
\end{proof}

Proposition \ref{trap} shows that when at most $M+1$ vacancies are available, a player starting on one side of an $M$-bridge can never reach the opposite side. We now use this obstruction to prove that no player can visit all vertices of the graph, which yields the desired upper bound on the Cartesian $1$-capacity.

\begin{theorem}\label{bridgebound}
	Let $G$ be a connected graph with $n$ vertices and an $M$-bridge $B_M=v_0v_1\ldots v_M$. If $F=\{f_1,\ldots,f_p\}$ is a patient collision-free $\ell$-march such that $p \geq n - (M+1)$, then $\orb(F,f_s(0))\not=V(G)$ for all $s\in\{1,\ldots,p\}$.    
\end{theorem} 
\begin{proof} 
Since $\deg(v_0)\geq 2$, there exists $x\in V(G)$ such that $\lev_{v_0}(x)=-1$. Similarly, since $\deg(v_M)\geq 2$, there exists $y\in V(G)$ such that $\lev_{v_0}(y)=M+1$. It follows from Proposition \ref{trap} that for any $s\in\{1,\ldots, p\}$, if $f_s(t)=x$, for some $t\in\{0,\ldots, \ell\}$, then $f_s(t')\not=y$ for all $t'\in\{0,\ldots,\ell\}$. Hence, $\orb(F,f_s(0))\not=V(G)$.
\end{proof}

\begin{corollary}\label{bridgecor} If a connected graph $G$ on $n$ vertices has an $M$-bridge, then $\ccap{1}{G}\leq n-\brig(G)-2$.
\end{corollary}
\begin{proof}
    By Theorem \ref{bridgebound}, if $p \geq n - M - 1$ and $F=\{f_1,\ldots,f_p\}$ is a patient collision-free $\ell$-march, then for all $i\in\{1, \ldots, p\}$, $\orb(F, f_i(0))\not=V(G)$, hence $F$ is not a patient collision-free $\ell$-tour. Thus, $p< n - M - 1$, for every $M$ for which there is an $M$-bridge, including $M=\brig(G)$.
\end{proof}

The bound from Corollary \ref{bridgecor} is tight, as demonstrated, for instance, by the double star $S_{2,2}$, obtained by joining the centres of two copies of $K_{1,2}$. Indeed, for this graph we have $n=6$, $\brig(S_{2,2})=1$, and $\ccap{1}{S_{2,2}}=3$.

\section{Trees}\label{sec:trees}

In this section, we specialize the general results of the previous section to trees. Since every edge of a tree is a bridge, the notions of branch-vertex density and $M$-bridges become particularly effective tools for analysing Cartesian $1$-capacity. By combining the lower bound from Section \ref{low_Car} with the upper bound from Section \ref{upp_Car}, we derive an exact characterization of the Cartesian $1$-capacity of trees.

\begin{proposition}\label{densebridge}
	Suppose that $G$ is a connected branch-vertex $N$-dense graph, for some $N\geq 3$. If $G$ has an $M$-bridge for some integer $M\geq 1$, then $M\leq N-2$. 
\end{proposition}
\begin{proof}
	Suppose on the contrary that $G$ has a $(N-1)$-bridge $v_1v_2\ldots v_{N}$. Then $\deg(v_1)\not=1$, $\deg(v_{N})\not=1$ and $\deg(v_i)=2\not\geq 3$ for all $i\in\{2,\ldots,N-1\}$. Since $\deg(v_1)\not=1$, there exists a vertex $v_0\not \in \{v_1,v_2,\ldots,v_{N}\}$ that is adjacent to $v_1$. Moreover, $P=v_0v_1\ldots v_{N}$ is a path in $G$. By Definition \ref{def:N-dense} of branch-vertex $N$-dense graphs, $P$ must satisfy at least one of the conditions of the definition. Since $\deg(v_i)=2$ for all $i\in\{2,\ldots,N-1\}$, condition 1 is not satisfied. Since $\deg(v_1)\not=1$ and  $\deg(v_{N})\not=1$ condition 2 is also not satisfied. The assertion follows immediately.
\end{proof}

\begin{proposition}\label{equa}
	If $T$ is a branch-vertex $N$-dense tree for some $N\geq 2$, then $N=\brig(T)+2$.
\end{proposition}
\begin{proof} 
If $N=2$, then $T$ is a star graph and the assertion follows. 

Assume $N\geq 3$. Let $k := \brig(T)+2$. It follows from Proposition \ref{densebridge} that $N\geq k$. 

We prove $N\leq k$ by contradiction. Assume that $N\geq k+1$. This implies $T$ is not $k$-dense, so there exists a path $P = v_0v_1\ldots v_k$ of length $k$ failing both conditions of Definition \ref{def:N-dense}. Failing the first condition implies $\deg(v_i)=2$ for all $i\in \{2,\ldots,k-1\}$. Failing the second condition implies either $\deg(v_k)\geq 2$ or $\deg(v_1)=2$. 

We consider two cases. If $\deg(v_k) \geq 2$, then since $v_1$ is an internal vertex ($\deg(v_1) \geq 2$), the path $v_1\ldots v_k$ forms a $(k-1)$-bridge. If $\deg(v_1) = 2$ and $\deg(v_k) = 1$, consider the path $v_0v_1\ldots v_{k-1}$ of length $k-1$. If $\deg(v_0) = 1$, then the entire tree $T$ would be exactly the path $P$, but $T$ cannot be a path graph since it is branch-vertex $N$-dense for $N \geq 3$. Thus, $v_0$ must be adjacent with other vertices in $T$, meaning $\deg(v_0) \geq 2$. Furthermore, since $v_{k-1}$ is an internal vertex of $P$, $\deg(v_{k-1}) = 2$. All internal vertices $v_1, \ldots, v_{k-2}$ of the path also have a degree of 2. Therefore, $v_0v_1\ldots v_{k-1}$ forms a $(k-1)$-bridge.

In both cases, $T$ contains a bridge of length $k-1$. Therefore, $\brig(T) \geq k-1 = (\brig(T) + 2) - 1 = \brig(T) + 1$, which is a contradiction. Thus, our assumption that $N \geq k+1$ is false, and we conclude $N = k = \brig(T)+2$.
\end{proof}

Notice that any tree $T$ that is neither a path nor a star graph contains an $M$-bridge, for some $M \geq 1$. Moreover, paths of length at least three contain an $M$-bridge, and star graphs contain no $M$-bridges.

\begin{proposition}\label{prop:trees-are-N-dense}
    If $T$ is an arbitrary tree different from a path, then $T$ is branch-vertex $(\brig(T)+2)$-dense.
\end{proposition}

\begin{proof}
    If $T$ is a star graph $K_{1,k}$, for some integer $k\geq 3$, then $\brig(T)=0$ and it is trivial to verify that $T$ is branch-vertex $2$-dense.

    Suppose $T$ is different from a path or a star graph.
    Denote $M:=\brig(T)$ and $N:=\brig(T)+2$. By maximality of $M$ the tree $T$ contains no bridge of length $M+1=N-1$. Let $P=v_0v_1\ldots v_N$ be an arbitrary path in $T$. We will show that $P$ satisfies at least one condition of Definition \ref{def:N-dense} (branch-vertex $N$-dense graphs). 
    
    Suppose condition 1 is not satisfied, hence $\deg(v_i)\leq 2$ for every $i\in\{2,\ldots,N-1\}$. Moreover, these vertices belong to $P$, hence $\deg(v_i)=2$ for every $i\in\{2,\ldots,N-1\}$. We will show that condition 2 of Definition \ref{def:N-dense} must be true, namely $\deg(v_1)\geq 3$ and $\deg(v_N)=1$. 

    First, towards a contradiction, suppose that $\deg(v_N)\geq 2$. Consider the path $P'=v_1 v_2\ldots v_N$. The vertex $v_1$ is adjacent to $v_0$ and $v_2$, hence $\deg(v_1)\geq 2$. Since we have that $\deg(v_i)=2$ for every $i\in\{2,\ldots,N-1\}$, it follows that $P'$ is an $(N-1)$-bridge, i.e. $(M+1)$-bridge, which contradicts the maximality of $M$. Hence, $\deg(v_N)=1$.

    Second, towards a contradiction, suppose that $\deg(v_1) < 3$. Since $v_1$ is an internal vertex of the path $P$, it follows that $\deg(v_1)=2$. Consider the path $P'=v_0 v_1\ldots v_{N-1}$. Since $v_{N-1}$ is adjacent to $v_{N-2}$ and $v_{N}$, we have $\deg(v_{N-1})\geq 2$. Let us consider $v_0$ next. If $\deg(v_0)\geq 2$, then again we have that $P'$ is an $(N-1)$-bridge, i.e. $(M+1)$-bridge, which contradicts the maximality of $M$. If $\deg(v_0)=1$, it follows from the fact that $\deg(v_N)=1$ and that $\deg(v_i)=2$ for every $i\in\{1,\ldots,N-1\}$ that the entire component of $T$ is exactly the path $P'$, contradicting the assumption that $T$ is not a path. Hence, $\deg(v_1)\geq 3$.

    It follows that for $N=M+2$ at least one of the conditions of Definition \ref{def:N-dense} is satisfied, hence $T$ is branch-vertex $N'$-dense for some $N'\geq 2$. Using Proposition \ref{equa} the assertion follows immediately.
\end{proof}

In the next lemma we will need the following result.

\begin{theorem}[\cite{gra26spancapacities}]\label{thm:c-topful}
    For connected graph $G$ on more than three vertices $\ccap{1}{G}=n-1$ if and only if it has no cut vertices.
\end{theorem}

\begin{lemma}\label{lem:stars}
    If $T$ is a star graph $K_{1,k}$, for some integer $k\geq 3$, then $\ccap{1}{T}=k-1$. 
\end{lemma}

\begin{proof}
    For any connected graph $G$ it holds that $\ccap{1}{G}\leq |V(G)|-1$ (see \cite{gra26spancapacities}). Note, $T$ has exactly $k+1$ vertices. Since $T$ has a cut-vertex, by Theorem \ref{thm:c-topful}, it holds that $\ccap{1}{T}<k$. 

    To complete the proof, we must show that $\ccap{1}{T} \geq k-1$. We will do this by constructing a patient collision-free tour $F=\{f_1, \ldots, f_{k-1}\}$ on $T$. 
    
    Let $V(T) = \{c, v_0, v_1, \ldots, v_{k-1}\}$, where $c$ is the central branch-vertex with $\deg(c) = k$, and $v_0, \ldots, v_{k-1}$ are the leaves. Let the initial configuration at time $t=0$ be $f_i(0) = v_i$ for all $i \in \{1, \ldots, k-1\}$. The initial set of vacant vertices is therefore $\{c, v_0\}$.
        
    Let $\ell=2(k-1)^2$. We describe an $\ell$-tour $F$ by constructing its dual $D_F=\{g_1,g_2\}$. For each $t \in \{0, \dots, \ell\}$, let
    \[
    g_1(t)=\begin{cases}
       c, & \text{if } t\equiv 0 \pmod 4,\\
       v_{(2\left\lfloor\frac{t}{4}\right\rfloor+1) \bmod k}, & \text{otherwise,} 
    \end{cases}
    \]
    and 
    \[
    g_2(t)=\begin{cases}
       c, & \text{if } t\equiv 2 \pmod 4,\\
       v_{(2\left\lfloor\frac{t+1}{4}\right\rfloor) \bmod k}, & \text{otherwise.} 
    \end{cases}
    \]

   By construction, the sequence $D_F(t)$ systematically alternates a vacancy through the central vertex $c$ to adjacent leaves, ensuring exactly one vacancy moves along a valid edge of $T$ at each step. Moreover, it is clear from the definition of $D_F$ that it is a patient collision-free $\ell$-march. Let $F=\{f_1,\ldots, f_{k-1}\}$ be the dual of $D_F$. Since the dual of a patient march is unique, it is well-defined. Moreover, it is also patient and collision-free. 
    By the construction of $D_F$, the march $F$ systematically shifts every player ($f_i$) to an adjacent leaf every $2(k-1)$ steps. Specifically, for each $i \in \{1, \ldots, k-1\}$ and any integer $t' \geq 0$ such that $2(k-1)t' \le \ell$, it holds that $f_i(2(k-1)t') = v_{(i-t') \bmod k}$. Consequently, evaluated over the set $\{0, 1,\ldots, \ell\}$, the orbit of each function contains every leaf. Furthermore, because any transition of $f_i$ between distinct leaves requires evaluating to the vertex $c$, it follows that $c \in \orb(f_i, f_i(0))$. Therefore, $\orb(f_i, f_i(0)) = V(T)$ for all $i \in \{1, \ldots, k-1\}$, hence $F$ is a tour.
\end{proof}

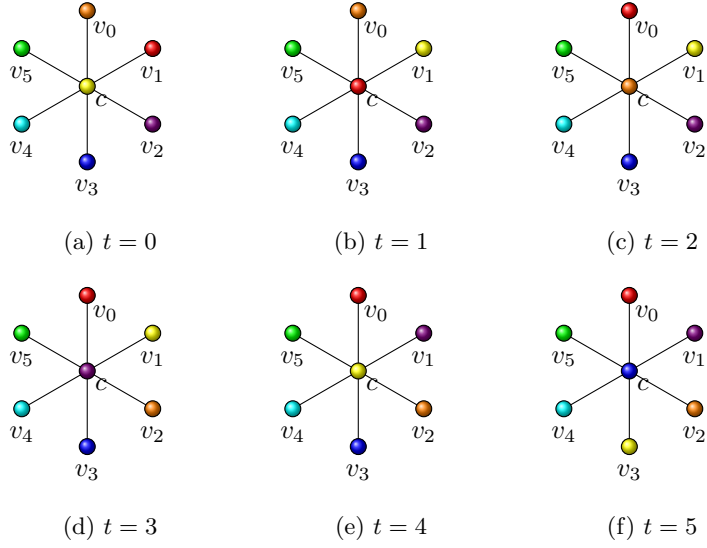
\begin{figure}
\centering
\begin{subfigure}{.2\textwidth}
\begin{tikzpicture} [scale=1]
\draw (0,0)--(90:1);
\draw (0,0)--(30:1);
\draw (0,0)--(-30:1);
\draw (0,0)--(-90:1);
\draw (0,0)--(-150:1);
\draw (0,0)--(150:1);
\filldraw [yellowball] (0,0) circle (3pt) node[below right] {$c$};
\filldraw [orangeball] (90:1) circle (3pt) node[below right] {$v_0$};
\filldraw [redball] (30:1) circle (3pt) node[below] {$v_1$};
\filldraw [violetball] (-30:1) circle (3pt) node[below] {$v_2$};
\filldraw [blueball] (-90:1) circle (3pt) node[below] {$v_3$};
\filldraw [cyanball] (-150:1) circle (3pt) node[below] {$v_4$};
\filldraw [greenball] (150:1) circle (3pt) node[below] {$v_5$};
\end{tikzpicture} 
\caption{$t=0$}
\end{subfigure}
\hspace{10pt}
\begin{subfigure}{.2\textwidth}
\begin{tikzpicture} [scale=1]
\draw (0,0)--(90:1);
\draw (0,0)--(30:1);
\draw (0,0)--(-30:1);
\draw (0,0)--(-90:1);
\draw (0,0)--(-150:1);
\draw (0,0)--(150:1);
\filldraw [redball] (0,0) circle (3pt) node[below right] {$c$};
\filldraw [orangeball] (90:1) circle (3pt) node[below right] {$v_0$};
\filldraw [yellowball] (30:1) circle (3pt) node[below] {$v_1$};
\filldraw [violetball] (-30:1) circle (3pt) node[below] {$v_2$};
\filldraw [blueball] (-90:1) circle (3pt) node[below] {$v_3$};
\filldraw [cyanball] (-150:1) circle (3pt) node[below] {$v_4$};
\filldraw [greenball] (150:1) circle (3pt) node[below] {$v_5$};
\end{tikzpicture} 
\caption{$t=1$}
\end{subfigure}
\hspace{10pt}
\begin{subfigure}{.2\textwidth}
\begin{tikzpicture} [scale=1]
\draw (0,0)--(90:1);
\draw (0,0)--(30:1);
\draw (0,0)--(-30:1);
\draw (0,0)--(-90:1);
\draw (0,0)--(-150:1);
\draw (0,0)--(150:1);
\filldraw [orangeball] (0,0) circle (3pt) node[below right] {$c$};
\filldraw [redball] (90:1) circle (3pt) node[below right] {$v_0$};
\filldraw [yellowball] (30:1) circle (3pt) node[below] {$v_1$};
\filldraw [violetball] (-30:1) circle (3pt) node[below] {$v_2$};
\filldraw [blueball] (-90:1) circle (3pt) node[below] {$v_3$};
\filldraw [cyanball] (-150:1) circle (3pt) node[below] {$v_4$};
\filldraw [greenball] (150:1) circle (3pt) node[below] {$v_5$};
\end{tikzpicture}
\caption{$t=2$}
\end{subfigure}\\
\vspace{10pt}
\begin{subfigure}{.2\textwidth}
\begin{tikzpicture} [scale=1]
\draw (0,0)--(90:1);
\draw (0,0)--(30:1);
\draw (0,0)--(-30:1);
\draw (0,0)--(-90:1);
\draw (0,0)--(-150:1);
\draw (0,0)--(150:1);
\filldraw [violetball] (0,0) circle (3pt) node[below right] {$c$};
\filldraw [redball] (90:1) circle (3pt) node[below right] {$v_0$};
\filldraw [yellowball] (30:1) circle (3pt) node[below] {$v_1$};
\filldraw [orangeball] (-30:1) circle (3pt) node[below] {$v_2$};
\filldraw [blueball] (-90:1) circle (3pt) node[below] {$v_3$};
\filldraw [cyanball] (-150:1) circle (3pt) node[below] {$v_4$};
\filldraw [greenball] (150:1) circle (3pt) node[below] {$v_5$};
\end{tikzpicture} 
\caption{$t=3$}
\end{subfigure}
\hspace{10pt}
\begin{subfigure}{.2\textwidth}
\begin{tikzpicture} [scale=1]
\draw (0,0)--(90:1);
\draw (0,0)--(30:1);
\draw (0,0)--(-30:1);
\draw (0,0)--(-90:1);
\draw (0,0)--(-150:1);
\draw (0,0)--(150:1);
\filldraw [yellowball] (0,0) circle (3pt) node[below right] {$c$};
\filldraw [redball] (90:1) circle (3pt) node[below right] {$v_0$};
\filldraw [violetball] (30:1) circle (3pt) node[below] {$v_1$};
\filldraw [orangeball] (-30:1) circle (3pt) node[below] {$v_2$};
\filldraw [blueball] (-90:1) circle (3pt) node[below] {$v_3$};
\filldraw [cyanball] (-150:1) circle (3pt) node[below] {$v_4$};
\filldraw [greenball] (150:1) circle (3pt) node[below] {$v_5$};
\end{tikzpicture} 
\caption{$t=4$}
\end{subfigure}
\hspace{10pt}
\begin{subfigure}{.2\textwidth}
\begin{tikzpicture} [scale=1]
\draw (0,0)--(90:1);
\draw (0,0)--(30:1);
\draw (0,0)--(-30:1);
\draw (0,0)--(-90:1);
\draw (0,0)--(-150:1);
\draw (0,0)--(150:1);
\filldraw [blueball] (0,0) circle (3pt) node[below right] {$c$};
\filldraw [redball] (90:1) circle (3pt) node[below right] {$v_0$};
\filldraw [violetball] (30:1) circle (3pt) node[below] {$v_1$};
\filldraw [orangeball] (-30:1) circle (3pt) node[below] {$v_2$};
\filldraw [yellowball] (-90:1) circle (3pt) node[below] {$v_3$};
\filldraw [cyanball] (-150:1) circle (3pt) node[below] {$v_4$};
\filldraw [greenball] (150:1) circle (3pt) node[below] {$v_5$};
\end{tikzpicture} 
\caption{$t=5$}
\end{subfigure}
\caption{An illustration of movement for $K_{1,6}$}
    \label{fig:star}
\end{figure}

An example of the first six stages of movement for $g_1$ (yellow) and $g_2$ (orange) in $K_{1,6}$ is given in Figure \ref{fig:star}. Their movement uniquely defines the movement of $f_1$ (red), $f_2$ (violet), $f_3$ (blue), $f_4$ (cyan) and $f_5$ (green).

Note that for paths $P$, it is known, \cite{BaTa23}, that $\cvSpan{P}=0$. Consequently, the Cartesian $d$-capacity is not defined for paths. Therefore, the following theorem covers all trees for which the Cartesian $d$-capacity is well defined.

\begin{theorem}\label{thm:1cc-trees}
    If $T$ is an $n$-vertex tree different from a path, then \[\ccap{1}{T}=n-\brig(T)-2.\]
\end{theorem}
\begin{proof} 
    Let $N = \brig(T) + 2$. 
    
    If $N=2$, then $T$ is a star graph $K_{1,k}$, for some integer $k\geq 3$, and $n=k+1$. By Lemma \ref{lem:stars} we have $\ccap{1}{T}=k-1 = (k+1) - (0+2) = n -\brig(T)-2$.
    
    Suppose now $N\geq 3$. By Proposition \ref{prop:trees-are-N-dense}, since $T$ is a tree different from a path, $T$ is branch-vertex $N$-dense. By Corollary \ref{lowerbound}, because $T$ is a branch-vertex $N$-dense graph, the capacity is bounded below by:
    \[\ccap{1}{T} \geq n - N = n - \brig(T)-2.\]   
    Furthermore, because the maximum bridge length in $T$ is exactly $\brig(T)$, it follows from Corollary \ref{bridgecor} that the capacity is bounded above by:
    \[\ccap{1}{T} \leq n - \brig(T) - 2 = n - \brig(T)-2.\]
    Since the upper and lower bounds coincide, $\ccap{1}{T} = n - \brig(T)-2$.
\end{proof}

Theorem \ref{thm:1cc-trees} is the foundation for a linear time algorithm that computes the Cartesian 1-capacity of an arbitrary tree $T$. Since this theorem requires one to determine the $\brig(T)$ in general case (when $T$ is not a path, nor a star graph), we provide the Algorithm \ref{alg:mbridge} that determines this invariant for such graphs.

\begin{algorithm}[!ht]
\caption{Compute $\brig(T)$}
\label{alg:mbridge}
\DontPrintSemicolon

\KwIn{A tree $T = (V, E)$ different from a path and a star}
\KwOut{$\brig{T}$}

\tcp{Step 1: Initialization}
\lFor{each vertex $v \in V$}{$\Visited[v] \gets \False$}
\MaxM $\gets 0$\;

\tcp{Step 2: Check base-case $M=1$}
\For{each edge $(u, v) \in E$}{
    \If{$\deg(u) \ge 3$ \textbf{and} $\deg(v) \ge 3$}{\MaxM $\gets \max(\MaxM, 1)$}
}

\tcp{Step 3: Looking for $M$-bridges, where $M\ge 2$}
\For{each vertex $v \in V$}{
    \If{$\deg(v) = 2$ \textbf{and} $\Visited[v] = \False$}{
        $k \gets 0$\; 
        \Boundary $\gets \emptyset$\;
        Initialize an empty Queue \Queue\;
        \Queue.\Enqueue{$v$}\; 
        $\Visited[v] \gets \True$\;
        
        \While{\Queue is not empty}{
            \Current $\gets$ \Queue.\Dequeue{}\;
            $k \gets k + 1$\;
            
            \For{each \Neighbor $\in N(\Current)$}{
                \If{$\deg(\Neighbor) = 2$ \textbf{and} $\Visited[\Neighbor] = \False$}{
                        $\Visited[\Neighbor] \gets \True$\;
                        \Queue.\Enqueue{\Neighbor}\;
                }
                \Else{\Boundary.\Append{\Neighbor}}
            }
        }
        
        \tcp{Step 4: Evaluate the length}
        $v_0 \gets \Boundary[1]$, $v_M \gets \Boundary[2]$\;
        \CurrentM $\gets 
        \begin{cases} 
            k + 1, & \text{if } \deg(v_0) \ge 3 \text{ and } \deg(v_M) \ge 3 \\
            k,     & \text{otherwise}
        \end{cases}$\;\label{line:a1-23}
        
        \MaxM $\gets \max(\MaxM, \CurrentM)$\;
    }
}

\Return \MaxM\;

\end{algorithm}

\begin{theorem}
    Given a tree $T$ on $n$ vertices that is neither a path nor a star graph, Algorithm \ref{alg:mbridge} correctly computes $\brig(T)$ in $\mathcal{O}(n)$ time.
\end{theorem}

\begin{proof}
    Let $M^*=\brig(T)$, and let $\MaxM$ be the value returned by the algorithm. We will show that $\MaxM = M^*$ by proving that every value $M$ computed by the algorithm corresponds to a valid $M$-bridge, and that no larger $M$-bridge exists in $T$.
    
    Because $T$ is a tree, every edge is a bridge. Thus, Condition 1 of Definition \ref{def:m-bridge} (definition of an $M$-bridge) is trivially satisfied for any path in $T$. Thus, the problem reduces entirely to verifying vertex degrees (Conditions 2 and 3).

    First, we will prove that $\MaxM \le M^*$. 

    The algorithm computes potential lengths in two distinct steps. In Step 2, it checks every edge $(u,v) \in E$. If $\deg(u) \ge 3$ and $\deg(v) \ge 3$, it records a length of $1$. Such an edge forms a path of length $M=1$ with zero internal vertices. Because both endpoints have a degree $\ge 2$, it is a valid $1$-bridge.

    In Step 3, the algorithm isolates maximal paths of degree-$2$ vertices (threads) using a Breadth-First Search. Let such a thread have $k$ internal vertices, forming a path $x_1, x_2, \ldots, x_k$. Let $v_0$ and $v_M$ be the boundary vertices adjacent to $x_1$ and $x_k$, respectively. Because $T$ is connected and not a path, $v_0$ and $v_M$ cannot both be leaves (degree $1$). Thus, exactly two cases exist for the boundary:
    \begin{description}
        \item[Case 1: Both $v_0$ and $v_M$ are branching vertices ($\deg \ge 3$).] The algorithm evaluates this in Step 4 as $\CurrentM = k+1$. The path $v_0 x_1 \ldots x_k v_M$ has exactly $k+1$ edges. All $k$ internal vertices have degree $2$ (satisfying Condition 2), and both endpoints have degree $\ge 3 \ge 2$ (satisfying Condition 3). This is a valid $(k+1)$-bridge.
        
        \item[Case 2: One boundary is a branching vertex, and the other is a leaf.] Without loss of generality, let $\deg(v_0) \ge 3$ and $\deg(v_M) = 1$. The algorithm falls to the \emph{otherwise} condition (line \ref{line:a1-23}) and evaluates $\CurrentM = k$. Because $v_M$ is a leaf, it cannot be the endpoint of an $M$-bridge. Hence, we obtain the path $v_0 x_1 \ldots x_{k-1} x_k$ with $k$ edges. The $k-1$ internal vertices have degree $2$, and the endpoints $v_0$ and $x_k$ have degrees $\ge 3$ and $2$, respectively. Both endpoint degrees are $\ge 2$, satisfying Condition 3. This is a valid $k$-bridge.
    \end{description}

    Since every evaluated $\CurrentM$ corresponds to a valid $\CurrentM$-bridge in $T$, the maximum found by the algorithm is $\MaxM \le M^*$.

    Second, we will prove that $\MaxM \ge M^*$. Because $T$ is neither a path nor a star, it has a diameter of at least $3$. Therefore, $T$ must contain at least one edge whose endpoints both have a degree $\ge 2$. By definition, this edge constitutes a valid $1$-bridge, hence $M^* \ge 1$. As established in the previous part of the proof, our algorithm evaluates these structures (either directly in Step 2 or via threads in Step 3), meaning it is guaranteed that $\MaxM \ge 1$.
    
    Suppose, towards a contradiction, that there exists an $M$-bridge $B$ of length $M > \MaxM$. Because $\MaxM \ge 1$, we have that  $M \ge 2$.
    
    Since $M \ge 2$, $B$ contains $M-1 \ge 1$ internal vertices, all of which must have a degree of $2$ in $T$. Therefore, these internal vertices must belong to some maximal thread of degree-$2$ vertices identified in Step 3. Let this maximal thread contain $k$ vertices. Since $B$ is a valid $M$-bridge, its endpoints cannot be leaves. To satisfy Condition 3, $B$ can only extend into the thread's boundary vertices if they have a degree of at least $3$. If a boundary vertex is a leaf, $B$ terminates at the last degree-$2$ vertex before it. Step 4 explicitly calculates the absolute maximum legal distance between non-leaf boundaries for this thread (yielding either $k+1$ or $k$). Therefore, $M$ cannot strictly exceed the value calculated for its containing thread, meaning $M \le \CurrentM \le \MaxM$, which contradicts our initial assumption.

    Finally, we need to prove the time complexity of the algorithm. Steps 1 and 2 evaluate the vertices and edges exactly once, taking $\mathcal{O}(n)$ time. In Step 3, the $\Visited$ array ensures that every degree-$2$ vertex is enqueued and processed at most once. The neighbors of degree-$2$ vertices are checked in $\mathcal{O}(1)$ time since their degree is strictly bounded. Thus, the entire algorithm runs in $\mathcal{O}(n)$ time.
\end{proof}

Now we state the final linear time algorithm that determines the Cartesian $1$-capacity of 
an arbitrary tree.

\begin{algorithm}[!ht]
\caption{Compute $\ccap{1}{T}$}
\label{alg:ccap1}
\DontPrintSemicolon

\KwIn{A tree $T = (V, E)$ on $n$ vertices}
\KwOut{The Cartesian capacity $\ccap{1}{T}$}

\BlankLine
\MaxDeg $\gets 0$\;

\BlankLine
\tcp{Step 1: Find the maximum degree}
\For{each vertex $v \in V$}{
    \lIf{$\deg(v) > \MaxDeg$}{\MaxDeg $\gets \deg(v)$}
}

\BlankLine
\tcp{Step 2: Evaluate topological constraints}
\If{\MaxDeg $\le 2$}{
    \Return \Undefined \tcp*{T is a path}
}
\ElseIf{\MaxDeg $= |V| - 1$}{
    \Return $\MaxDeg - 1$ \tcp*{T is a star graph $K_{1,k}, k\geq3$}
}
\Else{
    \Return $n-(\Brig{T}+2)$ \tcp*{General case processing}
}
\end{algorithm}

\begin{theorem}
    Given a tree $T$ on $n$ vertices, Algorithm \ref{alg:ccap1} correctly computes $\ccap{1}{T}$ in $\mathcal{O}(n)$ time.
\end{theorem}

\begin{proof}
    In Step 1, the maximum degree of the tree $T$, $\Delta(T)$, is computed and can be done in $\mathcal{O}(n)$ time. It is used to determine the structure of $T$, namely if $\Delta(T)\leq 2$ (can be checked in $\mathcal{O}(1)$ time), then $T$ is a path, hence in Step 2 the value returned is $\Undefined$. If $\Delta(T) = n-1$ and it is not a path (can also be checked in $\mathcal{O}(1)$ time), then $T$ is a star graph $K_{1,k}$, for some integer $k\geq 3$, hence, by Lemma \ref{lem:stars}, the value $\Delta(T)-1 = n-2$ is correctly returned. 
    Lastly, if $T$ is neither a path nor a star graph, then Algorithm \ref{alg:mbridge} is used to determine $\brig(T)$ and the correct value $n-(\Brig{T}+2)$ is returned (as proven in Theorem \ref{thm:1cc-trees}). Since this algorithm runs in $\mathcal{O}(n)$ time, the total complexity of Algorithm \ref{alg:ccap1} is also $\mathcal{O}(n)$.
\end{proof}

\section{Conclusion}\label{sec:conslusion}

Beyond their graph-theoretic significance, our results reveal strong connections with classical pebble-motion problems, multi-agent path finding, and token-swapping processes on trees. The obtained capacity bounds identify the precise structural obstacles responsible for routing complexity and offer a unifying topological perspective on several phenomena previously studied from an algorithmic viewpoint.

Several directions remain open for future research. Most notably, it would be interesting to extend the exact characterization of Cartesian $1$-capacity beyond trees, determine efficient algorithms for broader graph classes, and investigate Cartesian $d$-capacities for $d>1$. Understanding how graph capacities interact with other graph invariants and routing parameters may further strengthen the connection between graph topology and multi-agent movement problems.

\section*{Acknowledgments} 
Mateja Gra\v si\v c acknowledges the support of the Slovenian Research and Innovation  Agency (research core funding No. P1-0288). Aljoša Šubašić and Tanja Vojković were partialy supported by the ZMAJ project (IP-UNIST-45) funded by the European Union — NextGenerationEU and the VAL project (PK.3.4.17.0021) funded by the European Regional Development Fund (ERDF). Andrej Taranenko acknowledges the financial support from the Slovenian Research and Innovation Agency (research core funding No. P1-0297, projects N1-0285 and N1-0431).

\section*{Declaration of interests}
 
The authors declare that they have no conflict of interest. 

\section*{Data availability}
 
Our manuscript has no associated data.

\bibliographystyle{plain} 
\bibliography{bibliography}

\end{document}